\documentclass[12pt,english,reqno]{amsart}
\usepackage[latin9]{inputenc}
\usepackage{verbatim}
\usepackage{textcomp}
\usepackage{amsthm}
\usepackage{amstext}
\usepackage{amssymb}
\usepackage{graphicx}
\usepackage{esint}
\usepackage[hidelinks=true]{hyperref}
\date{\today}
\usepackage{mdef}

\makeatletter
\newcounter{cprop}[section]

\usepackage{amscd}\usepackage{amsthm}\usepackage[english]{babel}
\usepackage[arrow,matrix]{xy}\usepackage{cite}\usepackage{bbm}
\@ifundefined{definecolor}
 {\usepackage{color}}{}
\usepackage{MnSymbol}

\renewcommand{\Lip}{\operatorname{Lip}}

\renewcommand{\P}{ {\mathbb{P}} }

\newtheorem{theorem}[cprop]{Theorem}
\newtheorem*{theorem*}{Theorem}
\theoremstyle{plain}
\newtheorem*{ackno*}{Acknowledgements}

\newtheorem{assumption}[cprop]{Assumption}

\newtheorem{corollary}[cprop]{Corollary}

\newtheorem{example}[cprop]{Example}

\newtheorem{lemma}[cprop]{Lemma}

\newtheorem{proposition}[cprop]{Proposition}
\newtheorem{remark}[cprop]{Remark}

\numberwithin{equation}{section}
\usepackage[normalem]{ulem}

\definecolor{gray}{rgb}{0.75, 0.75, 0.75}
\newcommand{\todo}[1]{}
\newcommand{\gray}[1]{}
\newcommand{\blue}[1]{#1}
\newcommand{\revP}[1]{#1}
\newcommand{\revB}[1]{#1}

\begin{document}
\def\cprime{$'$}
\def\cprime{$'$}

\title{Regularization by noise for stochastic Hamilton-Jacobi equations}
\begin{abstract}
We study regularizing effects of nonlinear stochastic perturbations for fully nonlinear PDE. More precisely, path-by-path $L^{\infty}$ bounds for the second derivative of solutions to such PDE are shown. These bounds are expressed as solutions to reflected SDE and are shown to be optimal.
\end{abstract}

\author[P. Gassiat]{Paul Gassiat}
\address{Ceremade, Universit\'e de Paris-Dauphine\\
Place du Mar\'echal-de-Lattre-de-Tassigny\\
75775 Paris cedex 16, France}
\email{gassiat@ceremade.dauphine.fr}

\author[B. Gess]{Benjamin Gess}
\address{Max-Planck Institute for Mathematics in the Sciences \\
04103 Leipzig, Germany, Fakultät für Mathematik, Universität Bielefeld, D-33501 Bielefeld, Germany}
\email{bgess@mis.mpg.de}



\maketitle

\section{Introduction}

The purpose of this paper is to provide sharp, pathwise estimates for the $L^{\infty}$ norm of the second derivative of solutions to a class of SPDE of the type 
\begin{equation} \label{eq:spde}
du+\frac{1}{2}|Du|^{2}\circ d\xi_{t}=F(x,u,Du,D^{2}u)\, dt \quad\text{on }\mathbb{R}^{N},
\end{equation}
for $F$ satisfying appropriate assumptions detailed below, $\xi$ being a continuous function and  initial condition $u_0 \in BUC(\R^N)$. \revB{More precisely, under these assumptions we show that, for each $t\ge0$,
\begin{equation}\label{eq:intro-main_bd}
  \|D^2 u(t,\cdot)\|_{L^\infty} \leq \frac{1}{L^+(t)\wedge L^-(t)},
\end{equation}
where $L^\pm$ is the maximal continuous solution on $[0,\infty)$ to
\begin{equation}\begin{aligned}
dL^\pm(t) &= V_F(L^\pm(t)) dt \pm d\xi(t)  \mbox{ on } \{t\ge0 : L^\pm(t)>0\}, \;\; L^\pm \geq 0, \\
L^\pm(0) &= \frac{1}{\|D^2u_0\|_{L^\infty}} 
\label{eq:intro-reflected}
\end{aligned}
\end{equation}
and $V_F: \R_+ \to \R$ is a mapping depending only on $F$ (see Corollary \ref{cor:two-sided_bounds} below for the details).

While one-sided (i.e.\ semiconcavity or semiconvexity) bounds for the second derivative are typical for solutions of deterministic Hamilton-Jacobi-Bellman equations (cf. \cite{CS04,FS06}), two-sided (i.e.\ $C^{1,1}$) bounds in general do not hold for degenerate parabolic equations\footnote{See however the one-dimensional example in \cite{J04}.}. This is reflected by either $L^+$ or $L^-$ in \eqref{eq:intro-main_bd}, \eqref{eq:intro-reflected} with $\xi \equiv 0$ attaining zero value in finite time and then staying zero for all time. In contrast, we show that in the case of \eqref{eq:spde} such two-sided bounds may be obtained, due to the "stochastic" (or "rough") nature of the signal $\xi$. In particular, the inclusion of the random perturbation in \eqref{eq:spde} and consequently in \eqref{eq:intro-reflected} can cause both solutions $L^\pm$ to become strictly positive even after previously attaining zero value, thus implying a two sided bound on the second derivative of $u$ via \eqref{eq:intro-main_bd}. In this sense, we observe a regularization by noise effect.

We next give a series of applications illustrating this effect (cf.\ Section \ref{sec:examples} below for the details).

\begin{theorem} 
  Consider the stochastic $p$-Laplace equation\footnote{Equations of this form arise as (simplified) models of fluctuating hydrodynamics of the zero range process about its hydrodynamic limit (cf.\ \cite{DirrStamatakisZimmer} and \eqref{eq_Dirr_2} below).}
  \begin{equation*} 
      d u + \frac{\sigma}{2}|\partial_x u|^2 \circ d\beta(t)=\revP{\frac{1}{m}} \partial_x(|\partial_x u|^{m-1}\partial_x u)\;dt \quad \text{on }\R,
   \end{equation*}
    with $m\ge 3$, $\sigma>0$, $\beta$ a Brownian motion and initial condition $u_0\in (BUC \cap W^{1,\infty})(\R)$ and set $R:=\left\|\partial_x u_0\right\|_{L^\infty}$. Then, for all \revP{$\sigma^2 > 2(m-1)(m-2) R^{m-3}$} and all $t>0$,
      $$\|\partial_{xx} u(t)\|_{L^\infty} < \infty \quad\mathbb{P}\text{-a.s.}.$$
\end{theorem}
In contrast, for $\sigma = 0$ and $t>0$ large enough one typically has $\|\partial_{xx} u(t)\|_{L^\infty} = \infty$.

This dependence of a regularizing effect of noise on the strength of the noise $\s$ seems to be observed here for the first time\footnote{In contrast, critical noise intensities regarding synchronization by noise have been observed before (cf.\ e.g.\ \cite{ACW83,FGS14,V18}).}. We prove the critical noise intensity to be optimal: In the case $m=3$ for $\sigma^2 \le 4$ we show (cf.\ Corollary \ref{cor:plp-optimal} below) that $\mathbb{P}\text{-a.s.}$,
 $$\|\partial_{xx} u(t)\|_{L^\infty} = \infty\quad \text{ for all $t>0$ large enough.}$$ 

In fact, for suitable initial conditions (cf.\ Section \ref{sec:optimality} below) we obtain the sharp equality
\begin{equation}\label{eq:intro_sharp}
\|\partial_{xx} u(t)\|_{L^\infty}=\frac{1}{L^+(t) \wedge L^-(t)},
\end{equation}
where $L^\pm$ are the solutions to the reflected (at $0^+$) SDE with dynamics on $(0,\infty)$ given by 
\begin{align*}
&dL^\pm  =-\frac{2}{L^\pm(t)}dt\pm \sigma d\beta_{t},\;\; L^\pm(0)  =\frac{1}{\|(\partial_{xx} u_{0})_\pm\|_{L^\infty}}.
\end{align*}
This implies the optimality of \eqref{eq:intro-main_bd}.


\begin{theorem}\label{thm:into-hyp}
  Consider hyperbolic SPDE of the form
  \begin{equation}
  du+\frac{1}{2}|Du|^{2}\circ d\beta_{t}^{H}=F(Du)\, dt\quad\text{on }\mathbb{R}^{N},\label{eq:hyper_SPDE}
  \end{equation}
  where $\beta^{H}$ is a fractional Brownian motion with Hurst parameter $H\in(0,1)$, $F\in C^2(\R^N)$, and $u(0,\cdot)=u_0\in (BUC \cap W^{1,\infty})(\R^N)$. Then, for all $t>0$,
   \[
  \mathbb{P}(\|D^{2}u(t,\cdot)\|_{L^{\infty}}<\infty)=1,
  \]
  for $u$ being a solution to \eqref{eq:hyper_SPDE}.
\end{theorem}
%
In contrast, the solutions to the deterministic counterpart
\[
\partial_{t}w+\frac{1}{2}|Dw|^{2}=F(Dw) \quad \mbox{or}\quad \partial_{t}w=F(Dw)\quad\text{on }\mathbb{R}^{N}
\]typically develop singularities in terms of shocks of the derivative, that is, $Dw$ will become discontinuous for large times, even if $w_{0}$ is smooth. }

\revB{
The following particularly simple example may help to illustrate the regularizing effect of noise observed in this work \revP{(note that the bound does not depend on the regularity of the initial condition)}.

\begin{example} 
  Consider hyperbolic SPDE of the form
  \begin{equation}
  du+\frac{1}{2}|Du|^{2}\circ d\xi_{t} = 0\,\quad\text{on }\mathbb{R}^{N}, \label{eq:stoch_burgers}
  \end{equation}
  with $\xi \in C(\R_+)$ and $u(0,\cdot)=u_0\in BUC(\R^N)$.  Then
   \[
    \|D^{2}u(t,\cdot)\|_{L^{\infty}} \le \frac{1}{L^+(t)\wedge L^-(t)},
  \]
  where 
  $ L^+(t)= \xi_t - \min_{s\in [0,t]} \xi_s$, $L^-(t)=\max_{s\in [0,t]}\xi_s - \xi_t$.
\end{example}
}
\gray{
In contrast, in the deterministic case $\xi(t)=t$, the solution $u$ to \eqref{eq:stoch_burgers} typically develops shocks in finite time, that is, the second derivative becomes infinite and stays infinite at all later times.
}

\gray{
Our results may also be applied to some cases where, unlike in the previous examples, the deterministic part of the equation has a regularizing effect. For example, consider the equation
\[
\partial_{t}w=\frac{1}{2}(1-(\partial_{x} w)^2)\partial_{xx}w,\quad\text{on }\mathbb{R} , \;\;w(0,\cdot)=w^0
\]
with initial condition $w^0$ such that $\left\|\partial_x w^0\right\|_{L^\infty} < 1$. Since this is preserved by the equation, that is $\left\|\partial_x w(t,\cdot)\right\|_{L^\infty} < 1$ for all $t\ge 0$, the deterministic part is uniformly elliptic. In particular, the solutions are smooth at positive times. Our result yields that this is still true for the solution $u$ to
\[
du+\frac{1}{2}(\partial_{x}u)^{2}\circ d\xi_{t}=\frac{1}{2}\left( 1 - (\partial_x u)^2\right)  (\partial_{xx} u)  dt,\quad\text{on }\mathbb{R}, \;\;\; u(0,\cdot)=w^0
\]
 if the intensity is the noise is small enough. More precisely, if $\xi \in C^{\alpha}, \alpha > \frac{1}{2}$ or $\xi= \sigma B$ with $B$ a Brownian motion and $\sigma<1$, then (almost surely in the latter case)
$$\forall t >0, \;\;\;\left\|\partial_{xx} u(t,\cdot)\right\|_{L^\infty} < \infty.$$
Again this follows from properties of SDE, namely that the solutions to
$$dL^{\pm} = \frac{1}{L^{\pm}} dt \pm d\beta(t)
$$
do not hit $0$ at positive times.}

Finally, let us mention that our regularity results imply some estimates for large time behavior. For instance, if $u$ is a solution to the stochastic Hamilton-Jacobi equation
$$du+\frac{1}{2}(\partial_{x}u)^{2}\circ d\beta_{t}=0,\quad u(0,\cdot)=u^0(\cdot),$$ 
then, for all $t \geq 0$, (cf.\ Proposition \ref{prop:lt} below)
\begin{equation*} 
\left\|Du(t,\cdot)\right\|_{L^\infty} \leq \sqrt{\frac{2 \left\|u^0\right\|_{L^\infty}}{\max_{0\leq s \leq t}\beta(s) - \inf_{0\leq s \leq t} \beta(s)}}.
\end{equation*}
Note that when $\beta$ is a Brownian motion, we get a rate of decay in $t^{-1/4}$ which is the same rate as obtained in \cite{GS14-2}.

The proof of the main abstract result is based on the regularizing effects of the semi-groups $S_H$ and $S_{-H}$ associated to the Hamiltonians $H:=p\mapsto \frac{1}{2}p^2$ and $-H$. It is well-known that $S_H$ and $S_{-H}$ allow to obtain one-sided bounds (of the opposite sign) on the second derivative (cf e.g. \cite{L82}), and the fact that one can combine these two bounds to obtain $C^{1,1}$ bounds goes back to Lasry and Lions \cite{LL86}. Our main theorem is in a sense a generalization of their result.  

\gray{
Before stating our theorem in detail let us first consider some concrete examples (cf.\ Section \ref{sec:examples} below for details).}

\subsection{Literature}

The questions of regularizing effects and well-posedness by noise for (stochastic) partial differential equations have attracted much interest in recent years. The principle idea is that the inclusion of stochastic perturbations may lead to more regular solutions and in some cases even to the uniqueness of solutions. 
Historically, possible regularizing effects of additive noise have been investigated, e.g. for (stochastic) reaction diffusion equations
\[
dv=\Delta v\, dt+f(v)\, dt+dW_{t}
\]
in \cite{GP93} and for Navier-Stokes equations in \cite{FR02,FR08}. In \cite{FGP10,FF13,BFGM14}, well-posedness and regularization by linear multiplicative noise for transport equations, that is, for 
\[
dv=b(x)\nabla_{x}v\, dt+\nabla v\circ d\beta_{t},
\]
have been obtained. \revB{Regularization by noise phenomena have been observed in several classes of nonlinear PDE, such as Navier-Stokes equations \cite{FR02,FR08}, nonlinear Schrödinger equations \cite{DT11}, alpha-models of turbulence \cite{BBF14}, dyadic models for turbulence \cite{F11}, nonlinear heat equations \cite{GP93,DPFRV16}, geometric PDE \cite{SY04,DLN01}, Vlasov-Poisson equations \cite{DFV14} and point vortex dynamics in 2D Euler equations \cite{FGP11}, among many more. We refer to \cite{F11,G17} for more details on the literature.}

Recently, regularizing effects of \textit{non-linear} stochastic perturbations in the setting of (stochastic) scalar conservation laws have been discovered in \cite{GS14-2}. In particular, in \cite{GS14-2} it has been shown that quasi-solutions to 
\begin{equation}
dv+\frac{1}{2}\partial_{x}v^{2}\circ d\beta_{t}=0\quad\text{on }\mathbb{T}\label{eq:stoch_Burgers}
\end{equation}
where $\mathbb{T}$ is the one-dimensional torus
, enjoy fractional Sobolev regularity of the order 
\begin{equation}
v\in L^1([0,T];W^{\alpha,1}(\mathbb{T}))\quad\text{for all }\alpha<\frac{1}{2},\,\mathbb{P}\text{-a.s.}\label{eq:stoch_est}
\end{equation}
This is in contrast to the deterministic case, in which examples of quasi-solutions to
\[
\partial_{t}v+\frac{1}{2}\partial_{x}v^{2}=0\quad\text{on }\mathbb{T}
\]
have been given in \cite{DLW03} such that, for all $\alpha>\frac{1}{3}$,
\[
v \not\in L^1([0,T];W^{\alpha,1}(\mathbb{T})).
\]
In this sense, the stochastic perturbation introduced in \eqref{eq:stoch_Burgers} has a regularizing effect. In \cite{GS14-2}, the question of optimality of the estimate \eqref{eq:stoch_est} remained open.

Subsequently, the results and techniques developed in \cite{GS14-2} have been (partially) extended in \cite{GS16} to a class of parabolic-hyperbolic SPDE, as a particular example including the SPDE
\begin{equation}
dv+ \frac{1}{2}\partial_{x}v^{2}\circ d\beta_{t}=\frac{1}{12}\partial_{xx}v^{3}\, dt\quad\text{on }\mathbb{T}.\label{eq:stoch_PME}
\end{equation}
\blue{
Equations of the type \eqref{eq:stoch_PME} arise as (simplified) models of fluctuating hydrodynamics of the zero range process about its hydrodynamic limit, as informally shown by Dirr, Stamatakis, and Zimmer in \cite{DirrStamatakisZimmer}. More precisely, in \cite{DirrStamatakisZimmer} the fluctuations were shown to satisfy a stochastic nonlinear diffusion equation of the type
\begin{equation}\label{eq_Dirr_2}d v=\Delta\left(\Phi(v)\right)dt+\nabla\cdot\left(\sqrt{\epsilon\Phi(v)}\circ dW\right),\end{equation}
where $dW$ is space-time white noise. In the porous medium case $\Phi(\rho)=\rho|\rho|^{m-1}$, choosing $m=4$ and replacing $dW$ by spatially homogeneous noise, this becomes (up to constants)
\begin{equation*} 
 d v^\epsilon=\partial_{xx} (v|v|^3)+\frac{1}{2} \partial_x v^2\circ d\b_t.
\end{equation*}
}
In \cite{GS16}, the regularity of solutions to \eqref{eq:stoch_PME} was analyzed. More precisely, it was shown that 
\begin{equation*}
v \in L^1([0,T];W^{\alpha,1}(\mathbb{T}))\quad\text{for all }\alpha<\frac{2}{5},\,\mathbb{P}\text{-a.s}. 
\end{equation*}
However, neither optimality of these results nor regularization by noise could be observed in this case. That is, the regularity estimates for solutions to \eqref{eq:stoch_PME} proven in \cite{GS16} did not exceed the known regularity for the solutions to the non-perturbed cases
\begin{equation*}
\partial_{t}v+\frac{1}{2}\partial_{x}v^{2}=\frac{1}{12}\partial_{xx}v^{3}\quad\text{or}\quad\partial_{t}v=\frac{1}{12}\partial_{xx}v^{3}\text{ on }\mathbb{T}. 
\end{equation*}

In \cite{GS14-2,GS16} the estimation of the regularity of solutions to \eqref{eq:stoch_Burgers}, \eqref{eq:stoch_PME} relied on properties of the law of Brownian motion. The question of the {\it path-by-path} properties of $\b$ leading to regularization by noise could thus not be answered (cf.\ \cite{CG16} for related questions in the case of linear transport equations). 

\revB{If $u$ is the unique viscosity solution to the SPDE
\[
du+\frac{\sigma}{2}(\partial_{x}u)^{2}\circ d\beta_{t}=\frac{1}{12}\partial_{x}(\partial_{x}u)^{3} dt,\quad\text{on }\mathbb{R},
\]
then, informally, $v=\partial_{x}u$ is a solution to \eqref{eq:stoch_PME}. Hence, in the present work both the question of optimal regularity estimates for \eqref{eq:stoch_PME}, as well as an analysis of path-by-path properties of the driving noise leading to regularizing effects are addressed.}

\subsection{Organization of the paper} In Section \ref{sec:main_result} we give the precise statement of the assumptions and the main abstract theorem. Subsequently, we provide a series of applications of the main abstract result to specific SPDE in Section \ref{sec:examples}. The proof of the main abstract result is given in Section \ref{sec:proof}, while sufficient conditions for its assumptions are presented in Section \ref{sec:semiconv}. The proof of optimality is given in Section \ref{sec:optimality}. In the Appendix \ref{sec:visc_soln} we recall the employed well-posedness and stability results for stochastic viscosity solutions. 

\subsection{Notation}

We let $\R_+:=[0,\infty)$ and $S^N$ be the set of all symmetric $N\times N$ matrices. We further define $C_0^k([0, T ]; \R) := \{\xi \in C^k ([0, T ]; \R) : \xi(0) = 0\}$, $\Lip_{loc}(\R^N)$ to be the space of all locally Lipschitz continuous functions on $\R^N$ and $\Lip_b(\R_+)$ to be the space of all bounded Lipschitz continuous functions on $\R_+$. For a c\`adl\`ag path $\xi$ we set $\xi_{s,t}:=\xi_{t}-\xi_{s-}$.

\revP{
Given a continuous function $F$ we let $\left(S_F(s,t)\right)_{s \leq t}$ be the (two-parameter) semigroup, in the sense of viscosity solutions and in case it exists, for the PDE
\begin{equation} \label{eq:Fnot}
 \partial_t v = F(t,x,v,Dv,D^2v),
\end{equation}
namely if $v$ is a solution to \eqref{eq:Fnot} with $v(s,\cdot) =v_s$ then $S_F(s,t;v_s) = v(t,\cdot)$.
Similarly for a given $H$ we let $\left(S_H(t)\right)_{t \geq 0} =\left(S_H(0,t)\right)_{t \geq 0}$ be the (one-parameter) semigroup associated to the equation
  $$\partial_t v + H(Dv)= 0.$$
For a locally Lipschitz continuous function $V : (0,\infty) \to \R$  we define $\vp^V(t):\R_+\to \bar{\R}_+$, as the solution flow to the ODE $\dot{\ell}(t) = V(\ell)$ stopped when reaching the boundaries $0$ or $+\infty$ (i.e. $t\mapsto \vp^V(t;\ell)$ is the solution to this ODE with initial condition $\vp^V(0;\ell)=\ell$).}

For notational convenience, we set $H(p):=\frac{1}{2}|p|^2$ and  $S_H(-\d):=S_{-H}(\d)$ for $\d\ge 0$.

A modulus of continuity is a nondecreasing, subadditive function $\omega: [0, \infty) \to [0, \infty)$ such that $\lim_{r\to 0} \omega (r) = \omega(0) = 0.$ We define $UC(\R^N)$ to be the space of all uniformly continuous functions, that is, $u\in UC(\R^N)$ if $|u (x) - u (y)| \le \omega(|x-y|)$ for some modulus of continuity $\omega$. If, in addition, $u$ is bounded, we say $u\in BUC(\R^N)$. Furthermore, $USC(\R^N)$ (resp.\ $LSC(\R^N)$) denotes the
set of all upper- (resp. lower) semicontinuous functions in $\R^N$, and $BUSC(\R^N)$ (resp. $BLSC(\R^N)$) is the set of all bounded functions in $USC(\R^N)$ (resp. $LSC(\R^N)$).

\revP{We denote by $\|u\|_\infty$ the usual supremum norm of a function $u:\R^N \to \R$. For $E \subset \R^N$ we let $\|u\|_{L^\infty(E)} = \sup_{x \in E} \left|u(x)\right|.$ We further let $\|Du\|_{\infty}$ be the Lipschitz constant of $u$.}

We say that a function $u : \R^N \to \R$ is semiconvex (resp. semiconcave) of order $C$ if $x \mapsto u(x) + \frac{1}{2}C |x|^2$ is convex (resp. $x \mapsto u(x) - \frac{1}{2}C |x|^2$ is concave). \revP{We let $\|D^2 u \|_{\infty}$ be the smallest $C$ such that $u$ is both semiconcave and semiconvex of order $C$.}

For $a, b \in \R$ we set $a \wedge b := \min(a, b)$, $a \vee b := \max(a, b)$, $a + := \max(a, 0)$ and $a- := \max(-a, 0)$. For $m\ge 1$, $u\in\R$ we define $u^{[m]}:=|u|^{m-1}u$. We let $K,\tilde K$ be generic constants that may change value from line to line.

\begin{ackno*}
The work of PG was supported by the ANR, via the project ANR-16-CE40- 0020-01. The work of BG was supported by the DFG through CRC 1283.
\end{ackno*}

\section{Main abstract result}\label{sec:main_result}
%
%
%
%
%

We consider rough PDE of the form 
\begin{equation}\label{eq:main-SPDE}\begin{aligned}
   d u+  \frac{1}{2} |D u|^2 \circ d{\xi}(t)  &=F(t,x,u,Du,D^2u)dt\\
   u(0) &= u_0,
\end{aligned}\end{equation}
where $u_0 \in BUC(\R^N)$, $\xi$ is a continuous path and $F$ satisfies the typical assumptions from the theory of viscosity solutions, that is,
\begin{assumption}\label{asn:F}
\begin{enumerate}
  \item Degenerate ellipticity: For all $X,Y\in S^N$, $X\le Y$ and all $(t,x,r,p) \in [0,T] \times  \R^N \times \R \times \R^N $,
    $$F(t,x,r,p,X)\le F(t,x,r,p,Y).$$
  \item Lipschitz continuity in $r$: There exists an $L>0$ such that
    $$ |F(t,x,r,p,X) - F(t,x,s,p,X) |\le L|r-s|\quad\forall (t,x,s,r,p,X) \in [0,T] \times  \R^N \times \R\times \R \times \R^N \times S^N. $$
  \item Boundedness in $(t,x)$:
    $$ \sup_{[0,T]\times\R^N} |F(\cdot,\cdot,0,0,0)| <\infty. $$
  \item Uniform continuity in $(t,x)$: For any $R>0$, 
    $$ F \text{ is uniformly continuous on } [0,T] \times  \R^N \times [-R,R] \times  B_R \times B_R. $$
  \item Joint continuity in $(X,p,x)$: For each $R>0$ there exists a modulus of continuity $\o_{F,R}$ such that, for all $\a\geq 1$ and uniformly in $t\in [0,T]$, $x,y \in \R^N$, $r\in [-R,R]$,
    $$ F(t,x,r,\a (x-y),X)- F(t,y,r,\a (x-y),Y) \le \o_{F,R}(\a |x-y|^2+|x-y|),$$
  for all $X,Y\in S^N$ such that
    $$ -3\a\left( \begin{array}{ll}
    I & 0 \\
    0 & I
    \end{array} \right) \le \left( \begin{array}{ll}
        X & 0 \\
        0 & -Y
        \end{array} \right)
    \le 3\a \left( \begin{array}{ll}
        I & -I \\
        -I & I
        \end{array} \right).$$    
\end{enumerate}
\end{assumption}
We refer to the Appendix \ref{sec:visc_soln} for an according well-posedness result for \eqref{eq:main-SPDE}. 

We will make the following assumption on $F$ : 
\begin{assumption} \label{asn:odeF}
There exists $V_F : (0,\infty) \to \R$, locally Lipschitz and bounded from above on $[1,\infty)$ such that 
%
for all $g \in BUC(\R^n)$, $t\geq 0$,  one has for all $\ell \geq 0$,
$$D^2 g \leq \ell^{-1} Id \;\;\Rightarrow \;\; D^2 (S_F(t, g)) \leq \frac{Id}{\vp^{V_F}(t;\ell)},$$
the inequalities being understood in the sense of distributions.
\end{assumption}
The above assumption yields a control on the rate of loss of semiconcavity for $S_F$. Note that $\vp^{V_F}$ may take the value $0$ and thus no preservation of semiconcavity is assumed. 

\begin{theorem}\label{thm:main} 
Let $u_0 \in BUC(\R^N)$, $\xi \in C(\R_+)$, suppose that Assumptions \ref{asn:F}, \ref{asn:odeF} are satisfied and let $u$ be the \revP{unique viscosity solution (as defined in Theorem \ref{thm:app_wp})} to
 \begin{equation*}\label{eq:main_PDE} 
\left\{\begin{array}{l}	
d u + \frac{1}{2} |D u|^2 \circ d\xi(t) = F(t,x,u,Du,D^2u)dt, \\u(0,\cdot) = u_0. \end{array}\right.
\end{equation*} 
Suppose that $D^2u_0 \leq \frac{Id}{\ell_0}$ for some $\ell_0 \in [0,\infty)$, in the sense of distributions. Then, for each $t \geq 0$,
\begin{equation}\label{eq:main_bd}
  D^2 u(t,\cdot) \leq \frac{Id}{L(t)},
\end{equation}
in the sense of distributions, where $L$ is the maximal continuous solution on $[0,\infty)$ to
\begin{equation}\begin{aligned}
dL(t) &= V_F(L(t)) dt + d\xi(t)  \mbox{ on } \{t\ge0 : L(t)>0\}, \;\; L \geq 0, \\
L(0) &= \ell_0.
\label{eq:reflected}
\end{aligned}
\end{equation}
\end{theorem}

%
The proof of Theorem \ref{thm:main} is given in Section \ref{sec:proof} below. 

\revB{
\begin{corollary}\label{cor:two-sided_bounds}
  Let $u_0 \in BUC(\R^N)$, $\xi \in C(\R_+)$ and suppose that Assumptions \ref{asn:F}, \ref{asn:odeF} are satisfied by $F^+:=F$ and $F^-(t,x,r,p,X) := -F(t,x,-r,-p,-X)$. 
  Let $u$ be the unique viscosity solution to \eqref{eq:main_PDE} and suppose that $-\frac{Id}{\ell_0^-} \leq D^2u_0 \leq \frac{Id}{\ell_0^+}$ for some $\ell_0^\pm \in [0,\infty)$, in the sense of distributions. Then, for each $t \geq 0$,
  \begin{equation*} 
     \|D^2 u(t,\cdot)\|_{\infty} \leq \frac{1}{L^+(t)\wedge L^-(t)},
  \end{equation*}
  in the sense of distributions, where $L^\pm$ is the maximal continuous solution to \eqref{eq:reflected} with initial value $\ell_0^\pm$, drift $V_{F^\pm}$ and driven by $\pm\xi$.
\end{corollary}}
This corollary follows from Theorem \ref{thm:main} applied to $u$ and $-u$. 
%

\section{Applications}\label{sec:examples}
In this section we provide a series of PDE for which regularization by noise can be observed based on our main abstract Theorem \ref{thm:main}. 

\revP{We first present a series of PDE to which Assumption \ref{asn:odeF} applies. We defer the proof of this fact (as well as the statement of a more general criterion) to Section \ref{sec:semiconv}.

\begin{proposition}\label{ex:LMcond1}
\begin{enumerate}
\item First-order PDE: Let
    $$F=F(t,x,p)  \in C([0,T];C^2_b(\R^N \times \R^N)).$$
 Then Assumption \ref{asn:odeF} is satisfied with
    $$V_F(\ell) =  - \left\| F_{xx} \right\|_\infty \ell^2  -  2\left\| F_{xp} \right\|_\infty \ell - \left\| F_{pp} \right\|_\infty.$$
 More generally, let $F=F(t,x,p)  \in C([0,T]\times \R^N \times \R^N)$ such that $(x,p) \mapsto F(t,x,p)$ is semiconcave of order $C_F$. Then, Assumption \ref{asn:odeF} is satisfied with 
        $$V_F(\ell) =  - C_F(1+\ell^2).$$
\item Quasilinear PDE: Let
             $$F(x,p,A) = Tr(a(x,p) A) \in C(\R^N \times\R^N \times S^N),$$
            where $a(x,p) \in C^2(\R^N\times \R^N)$ is nonnegative, has bounded second derivative 
            and $(y,p)\mapsto \sqrt{a(y,p)}$ is convex. Then Assumption \ref{asn:odeF} is satisfied with
             $$V_F(\ell) = - N  \left\| a_{xx} \right\|_\infty \ell -  2N \left\| a_{xp} \right\|_\infty - N   \left\| a_{pp} \right\|_\infty \frac{1}{\ell}.$$
\item Monotone, concave, fully nonlinear PDE: Let 
    $$F = F(t,A) \in C([0,T] \times S^N)$$
  be concave and non-decreasing in $A\in S^N$. Then Assumption \ref{asn:odeF} is satisfied with $V_F = 0$.  
\item One-dimensional, fully nonlinear PDE: Let $F=F(t,x,p,A)\in C([0,T]\times\R \times \R \times \R)$ 
such that $(x,p) \mapsto F(t,x,p,A)$ is semiconcave of order $C_F(A)$. Then, Assumption \ref{asn:odeF} is satisfied with
     $$V_F(\ell) =  - C_F(1+\ell^2).$$
\end{enumerate}
\end{proposition}
}

\begin{theorem}\label{thm:plp}
  We consider the quasilinear PDE
  \begin{equation*}\begin{aligned} 
    d u + \frac{1}{2}|Du|^2 \circ d\xi(t) &= a(D u)\D u\;dt\quad\text{on }[0,T]\times\R^N, \\
    u(0) &= u_0,
  \end{aligned}\end{equation*}
  where $u_0\in (BUC\cap W^{1,\infty})(\R^N)$, $a\in C^2(\R^N)$ is nonnegative such that $p\mapsto \sqrt{a(p)}$ is convex. Then,
    $$\|D^2u(t,\cdot)\|_\infty \le \frac{1}{L^+(t)\wedge L^-(t)},$$
  where $L^\pm$ are the maximal solutions on $\R_+$ to 
  \begin{align*}
    dL^+(t) &= -\frac{N\|a_{pp}\|_{L^\infty(B_{R}(0))}}{L^+(t)}+d\xi(t),\quad L^+(0) = \frac{1}{\|(D^2 u_0)_+\|_\infty},\\
    dL^-(t) &= -\frac{N\|a_{pp}\|_{L^\infty(B_{R}(0))}}{L^-(t)}-d\xi(t),\quad  L^-(0) = \frac{1}{\|(D^2 u_0)_-\|_\infty},
  \end{align*}
  with $R := \left\|D u_0\right\|_\infty$.
  
  In particular this includes the $p$-Laplace equation in one space dimension
    $$d u + \frac{1}{2}|\partial_x u|^2 \circ d\xi(t)= \frac{1}{m}\partial_x(\partial_x u)^{[m]}\;dt, $$
  with $a(p)=|p|^{m-1}$ and $m\ge 3$.
\end{theorem}
\begin{proof}
  We aim to apply Theorem \ref{thm:main}. Hence, we have to verify Assumption \ref{asn:odeF}.
 
 Fix $v_0$ in $BUC\cap W^{1,\infty})(\R^N)$ and let $v$ be the (unique bounded) viscosity solution to
  \begin{equation} \label{eq:aDelta}
    \partial_t v = a(D v)\D v, \;\;\;\;
     v(0) = v_0.
  \end{equation}
  Note that by Lemma \ref{lem:Lipvisc}, one has $\|Dv(t)\|_\infty \leq \|Dv_0\|_\infty$, so that modifying $a$ outside of the ball of radius $R$ does not change the solution to \eqref{eq:aDelta}, and we may assume that $\|a_{pp}\|_{L^\infty(\R^N)} = \|a_{pp}\|_{L^\infty(B_{R}(0))}$.
  
By Proposition \ref{ex:LMcond1} (2), Assumption \ref{asn:odeF} holds for both $F^+(p,A) = -Tr(a(p)A)$ and $F^-(p,A) = -Tr(a(-p)A)$  with both of $V_F^{\pm}$ given by
\[V(\ell) = -\frac{N\|a_{pp}\|_{L^\infty(B_{R}(0))}}{\ell}.\]
The result then follows from Corollary \ref{cor:two-sided_bounds}.
\end{proof}

\revB{
\begin{corollary} 
  Under the same assumptions on $a$ and $u_0$ as in Theorem \ref{thm:plp} consider the SPDE 
  \begin{equation*}\begin{aligned} 
      d u + \frac{\sigma}{2}|Du|^2 \circ d\beta(t) &= a(D u)\D u\;dt\quad\text{on }[0,T]\times\R^N, \\
      u(0) &= u_0,
    \end{aligned}\end{equation*}
    with $\sigma>0$ and $\beta$ a standard Brownian motion. Let $R := \left\|D u_0\right\|_\infty$. Then, if $\sigma^2 > 2N\|a_{pp}\|_{L^\infty(B_{R}(0))}$, $t>0$,
          $$\|D^2 u(t)\|_{\infty} < \infty \quad\mathbb{P}\text{-a.s.}.$$
\end{corollary}
\begin{proof} Immediate consequence of Theorem \ref{thm:plp} together with Proposition \ref{prop:reflection_for_BM} below.
\end{proof}}

\begin{theorem} 
  We consider the first-order PDE
    $$d u + \frac{1}{2}|Du|^2 \circ d\xi(t) = F(D u)dt\quad\text{on }\R^N,$$
  where $u_0\in (BUC \cap W^{1,\infty})(\R^N)$ and $F\in C^2(\R^N)$. Then, 
    \begin{equation*} 
       \|D^2 u(t,\cdot)\|_{L^{\infty}} \le \frac{1}{L^+(t)\wedge L^-(t)},
    \end{equation*}
    where $L^\pm$ are the maximal continuous solutions on $\R_+$ to 
   \begin{equation}\label{eq:HJ-bd-2}\begin{aligned}
       dL^+(t) &= -\|F_{pp}\|_{L^\infty(B_{R}(0))}dt+d\xi(t),\quad L^+(0) = \frac{1}{\|(D^2 u_0)_+\|_\infty},\\
       dL^-(t) &= -\|F_{pp}\|_{L^\infty(B_{R}(0))}dt-d\xi(t), \quad  L^-(0) = \frac{1}{\|(D^2 u_0)_-\|_\infty},
     \end{aligned}\end{equation}
    where $R=\left\|D u_0\right\|_\infty$. 
\end{theorem}
\begin{proof}
As in the proof of Theorem \ref{thm:plp}, this is a direct consequence of Corollary \ref{cor:two-sided_bounds} and of Proposition \ref{ex:LMcond1} (1).
\end{proof}

\revB{
The proof of Theorem \ref{thm:into-hyp} now follows from the fact that the solutions $L^\pm$ to \eqref{eq:HJ-bd-2} \revP{with initial condition $L^\pm(0)=0$} are given by
\begin{align*}
  L^+(t) &= (\xi(t) -\|F_{pp}\|_{L^\infty(B_{R}(0))}t) - \min_{s\in[0,t]} (\xi(s)-\|F_{pp}\|_{L^\infty(B_{R}(0))}s)\\
  L^-(t) &= \max_{s\in[0,t]} (\xi(s)+\|F_{pp}\|_{L^\infty(B_{R}(0))}s) - (\xi(t) +\|F_{pp}\|_{L^\infty(B_{R}(0))}t).
\end{align*}}
\revP{Then, if $\xi=\beta^H$ is a fractional Brownian motion then for all $t >0$ one has $\P$-a.s. that
\[ \limsup_{ s \uparrow t} \frac{\xi(t)-\xi(s)}{t-s} = \limsup_{ s \uparrow t} \frac{\xi(s)-\xi(t)}{t-s} = +\infty,\]
so that $L^+(t) \wedge L^-(t) >0$.
}
\begin{theorem}
  We consider the quasilinear, one-dimensional PDE
  \begin{equation*}\begin{aligned}
    \partial_t u +\frac{1}{2}|\partial_x u|^2\circ d\xi(t) &= F(\partial_{xx}u)dt,\\
    u(0) &=u_0 \in BUC(\R),
  \end{aligned}\end{equation*}
  where $F\in C^0(\R)$ is non-decreasing. Then, 
  \begin{equation}\label{eq:one-dim}
     \|\partial_{xx} u(t,\cdot)\|_{L^{\infty}} \le \frac{1}{L^+(t)\wedge L^-(t)},
  \end{equation}
  where 
    $$L^+(t)= \xi(t)-\min_{s\in [0,t]}\xi(s) ,\ L^-(t)= \max_{s\in [0,t]}\xi(s)-\xi(t) .$$
\end{theorem}
\begin{proof}
Note that the $L^{\pm}$  are the maximal continuous solutions to $dL^{\pm} = \pm d\xi$, $L^\pm \geq 0$, $L^\pm(0)=0$. The results is then immediate from Corollary \ref{cor:two-sided_bounds} and Proposition \ref{ex:LMcond1} (4).
\end{proof}

\begin{remark} 
  We emphasize that the estimate \eqref{eq:one-dim} is uniform in $F$ and $u_0$. For example, consider $F^m(r) := r^{[m]} = |r|^{m-1}r \to \sgn(r)$ for all $r\in\R$ for $m\to 0$ and let $u_0^m \in (BUC \cap W^{1,1})(\R)$ with $u_0^m \to u_0$ in $W^{1,1}(\R)$. Then, at least formally, \eqref{eq:one-dim} continues to hold for the limit
      $$d u  +\frac{1}{2}|\partial_x u|^2\circ d\xi(t) = \sgn(\partial_{xx}u)dt$$
    implying Lipschitz bounds for the stochastic total variation flow
      $$d v +\frac{1}{2}\partial_x v^2\circ d\xi(t) = \partial_{x}\sgn(\partial_{x}v)dt.$$
    These bounds improve the deterministic case. Indeed, in \cite[Section 2.5]{BF12} it has been shown that the solution $v(t,\cdot)$ to the total variation flow in one spatial dimension
  \begin{equation*} 
     \partial_t v = \partial_x\sgn(\partial_xv)
  \end{equation*}
 is a step-function if $v_0$ is. In particular, for $v_0 \in BV(\R)$ one only has $v(t)\in BV(\R)$ in general. 
\end{remark}

\begin{proposition} \label{prop:lt}
 Let $u$ be the solution to
\begin{equation}\begin{aligned}\label{eq:ltb}
    du +\frac{1}{2}|Du|^2\circ d\xi(t)  &= F(Du,D^2 u) dt,\\
    u(0) &=u_0 \in BUC(\R^N),
  \end{aligned}\end{equation}
 where $F$ satisfies the assumptions of Theorem \ref{thm:main}. Then for all $t \geq 0$
\begin{equation*} 
\left\|Du(t,\cdot) \right\|_{\infty} \leq \inf_{0 \leq s \leq t} \sqrt{\frac{2 \left( \sup u_0 - \inf u_0\right)}{L^+(s) \vee L^-(s)}}
\end{equation*}
where $L^\pm$ are the bounds on $D^2u$ from Theorem \ref{thm:main}.

\end{proposition}

\begin{proof}
This is an immediate consequence of Theorem \ref{thm:main}, noting that if $u$ is semiconcave (or semiconvex) of order $C$ then $\|Du\|_\infty \leq \sqrt{2 C \left( \sup u - \inf u\right)}$ (e.g. \cite[p.240]{L82}), and the fact that since the coefficients in \eqref{eq:ltb} only depend on $Du$ and $D^2u$, $\left( \sup u(t,\cdot) - \inf u(t,\cdot)\right)$ and $\|Du(t,\cdot)\|_\infty$ are nonincreasing in $t$ (cf. Lemma \ref{lem:Lipvisc}).
\end{proof}

\section{Proof of Theorem \ref{thm:main}}\label{sec:proof}

The proof of Theorem \ref{thm:main} is based on a Trotter-Kato splitting scheme for \eqref{eq:main-SPDE}.  The estimate \eqref{eq:main_bd} is then proven for the corresponding approximating solutions $u^n$ with respect to a discretization $L^n$ of $L$, based on semiconvexity estimates for $S_H$, with $H(p) = \frac{1}{2}|p|^2$. The corresponding estimates are derived in Section \ref{sec:semiconvex} below. The rest of the proof then consists in proving the convergence of the approximations $L^n$ (cf.\ Section \ref{sec:rSDE} below) and $u^n$ (cf.\ Section \ref{sec:trotter} below). Finally, the proof of Theorem \ref{thm:main} is given in Section \ref{sec:proof}.

\subsection{Inf- and sup-convolution estimates}\label{sec:semiconvex}

In this section we provide Lipschitz and semiconvexity estimates for $S_H$ with $H(p) = \frac{1}{2}|p|^2$. We refer to \cite{L82,LL86} for related arguments.

Recall that \revP{for $\phi\in BUC(\R^N)$,} $S_H(\delta,\phi)$ can be written as 
$$S_H(\delta,\phi)(x) =
\begin{cases}
  \sup_{y \in \R^N} \left( \phi(y) - \frac{|x-y|^2}{2\delta}\right),& \text{if }\d\ge 0 \\
  \inf_{y \in \R^N} \left( \phi(y) + \frac{|x-y|^2}{2|\delta|}\right),& \text{if }\d\le 0.
\end{cases}$$
\revP{
We then extend the definition of $S_H(\delta,\phi)$ to arbitrary $\phi: \R^N \to \R$ by the above formula ($S_H(\delta,\phi)$ may possibly take the values $+\infty$ or $-\infty$).
}

\begin{lemma} \label{lem:InfSup}
If $\phi:\R^N \to \R$ is convex (resp. concave), then so is $S_H(\delta,\phi)$, for all $\delta \in \R$. 
\end{lemma}
\begin{proof}
We will prove the claim only for $\delta>0$, the case $\delta <0$ then follows noting that $S_H(\delta,-\phi)=-S_H(-\delta,\phi)$.

We begin by the case when $\phi$ is concave. Then for any $x_1, x_2 \in \R^N$ and $\lambda \in [0,1]$,
\begin{align*}
&S_H(\delta,\phi)(\lambda x_1 + (1-\lambda) x_2) \\
&= \sup_{y\in\R^N} \left\{ \phi(y) - \frac{1}{2\delta} \left| y -(\lambda x_1 + (1-\lambda) x_2) \right|^2\right\} \\
&= \sup_{y_1,y_2\in\R^N} \left\{ \phi(\lambda y_1 + (1-\lambda) y_2) - \frac{1}{2\delta} \left| \lambda(y_1 -x_1) + (1-\lambda)(y_2- x_2) \right|^2\right\} \\
&\geq \lambda  \sup_{y_1\in\R^N} \left\{ \phi(y_1)   - \frac{1}{2\delta} \left|y_1 -x_1\right|^2 \right\} + (1-\lambda) \sup_{y_2\in\R^N}  \left\{\phi(y_2) - \frac{1}{2\delta} \left| y_2- x_2\right|^2\right\} \\
&= \lambda S_H(\delta,\phi)(x_1) + (1-\lambda) S_H(\delta,\phi)(x_2),
\end{align*}
where in the third inequality we have used the concavity of $\phi$ and of $ - 1/(2\delta) |\cdot|^2$.

We now assume that $\phi$ is convex. Then for $x_1, x_2 \in \R^N$ and $\lambda \in [0,1]$,
\begin{align*}
& S_H(\delta,\phi)(\lambda x_1 + (1-\lambda) x_2) 
\\&= \sup_{z\in\R^N} \left\{ \phi\left(\lambda x_1 + (1-\lambda) x_2 - z\right) - \frac{1}{2\delta} \left| z\right|^2\right\} \\
&\leq \sup_{z\in\R^N} \left\{ \lambda ( \phi(x_1 -z) - \frac{1}{2\delta} |z|^2 ) + (1-\lambda) (\phi(x_2 -z) - \frac{1}{2\delta} |z|^2 )\right\} \\
&\leq  \lambda \sup_{z\in\R^N} \left\{ \phi(x_1 -z) - \frac{1}{2\delta} |z|^2 \right\} + (1-\lambda)\sup_{z\in\R^N} \left\{ \phi(x_2 -z) - \frac{1}{2\delta} |z|^2\right\}  \\
&=  \lambda S_H(\delta,\phi)(x_1) + (1-\lambda) S_H(\delta,\phi)(x_2).
\end{align*}
\end{proof}

\begin{proposition} \label{prop:InfSup}
Let $\phi \in BUC(\R^N)$, $\psi=S_H(\phi,\delta)$ for some $\delta \in \R$ and
$\lambda \in [0,\infty)$. Then
\begin{equation} \label{eq:D2+}
D^2 \phi \leq \lambda^{-1} Id \;\; \Rightarrow D^2 \psi\leq (\lambda - \delta)^{-1}_+ Id,
\end{equation}
\begin{equation}\label{eq:D2-}
D^2 \phi \geq -  \lambda^{-1} Id \;\; \Rightarrow D^2  \psi \geq -(\lambda + \delta)^{-1}_+ Id. 
\end{equation}
\end{proposition}
\begin{proof} 

To prove \eqref{eq:D2+}, \eqref{eq:D2-}, we again may assume without loss of generality that $\delta >0$. We focus on \eqref{eq:D2-} namely we prove that if $\psi = S_H(\delta,\phi)$,
$$ \phi + \frac{1}{2\lambda}|\cdot|^2 \mbox{ convex } \Rightarrow  \psi + \frac{1}{2(\lambda + \delta)}|\cdot|^2 \mbox{ convex}.$$
Indeed,
\begin{eqnarray*}
\psi(x) +  \frac{1}{2(\lambda + \delta)}|x|^2  &=& \sup_{y \in \R^N} \left\{\phi(y) - \frac{1}{2\delta} |x-y|^2  +  \frac{1}{2(\lambda + \delta)}|x|^2\right\} \\
&=&\sup_{y \in \R^N} \left\{\phi(y) + \frac{1}{2\lambda} |y|^2 - \frac{1}{2\lambda} |y|^2 - \frac{1}{2\delta} |x-y|^2  +  \frac{1}{2(\lambda +\delta)}|x|^2\right\}.
\end{eqnarray*}
By a direct computation, $ \frac{1}{2\lambda} |y|^2 + \frac{1}{2\delta} |x-y|^2  - \frac{1}{2(\lambda +\delta)}|x|^2$ can be written as $\alpha |x- \beta y|^2$ for some $\alpha, \beta \geq 0$, so that (after an affine change of coordinates) one can apply Lemma \ref{lem:InfSup} to obtain convexity of $\psi + \frac{1}{2(\lambda + \delta)}|\cdot|^2$.

The proof of \eqref{eq:D2+} is similar (using the preservation of concavity from Lemma \ref{lem:InfSup}).
\end{proof}

%
%

\subsection{Reflected SDE}\label{sec:rSDE}

In this section we first study stability properties of solutions to reflected SDE and then their boundary behavior.

Let $V$ be locally Lipschitz on $(0,+\infty)$, bounded from above on $[1,\infty)$, and $\xi$ be a continuous path. In this section we study the maximal solution on $[0,T]$ to
\begin{equation} \begin{aligned}\label{eq:sde}
   dX(t) &= V(X(t)) dt + d\xi(t) \mbox{ on } \{X>0\},\ X \geq 0,\ X \mbox{ continuous}\\
   X(0) &= x \in \R_+.
\end{aligned}\end{equation}
More precisely, a function $X\in C([0,T];\R_+)$ is said to be a solution  to \eqref{eq:sde} if, for all $s\le t \in [0,T]$,
$$ X>0 \mbox{ on }[s,t] \;\; \Rightarrow \;\; X(t) = X(s) + \int_s^t V(X(u)) du + \xi_{s,t}.$$
Let $\mathcal{S}(V,\xi,x)$ be the set of solutions
. Note that by the assumptions on $V$ there exists a unique solution $X$ to \eqref{eq:sde} until $\tau = \inf\{t\ge0 : \lim_{s\uparrow t} X(s) =0\}$, and a particular element of $\mathcal{S}(V,\xi,x)$ is given by letting $X(t) \equiv 0$ for $t \geq \tau$.

\begin{proposition} \label{prop:Xhat}
Let $V$ be locally Lipschitz on $(0,+\infty)$, bounded from above on $[1,\infty)$, and $\xi$ be a continuous path.  Let
  $$\hat{X}(t) := \sup \left\{ Y(t) : Y \in \mathcal{S}(V,\xi,x)\right\} .$$
Then,
$\hat{X} \in \mathcal{S}(V,\xi,x)$.
\end{proposition}

\begin{proof}
We first show that elements of $\mathcal{S}(V,\xi,x)$ are equibounded and equicontinuous. Indeed, it is easy to see that
$$M:=  x+1+ T \left\|V_+\right\|_{L^\infty([1,+\infty))} + 2 \left\| \xi_{0,\cdot}\right\|_{L^\infty([0,T])}$$
is an upper bound for $\hat{X}$. Then letting for $\varepsilon >0$
$$\omega_\varepsilon(r):= r \left\|V\right\|_{L^\infty([\varepsilon,M])} + \omega^\xi(r)$$
where $\omega^\xi$ is a modulus of continuity for $\xi$ on $[0,T]$, one sees that each element $X$ of $\mathcal{S}(V,\xi,x)$ admits $\omega_\varepsilon$ as a modulus of continuity on (connected subsets of) $\left\{X\geq \varepsilon\right\}$. \revP{Now let
$$\omega(r):= \inf_{\varepsilon>0} \left( 2 \varepsilon + 2\omega_\varepsilon(r)\right)$$
and note that $\limsup_{r \to 0} \omega(r) \leq \inf_{\varepsilon> 0} \left( 2 \varepsilon + 2\omega_\varepsilon(0^+)\right) = 0$. We now claim that $\omega$ is a modulus of continuity for $X$.} Indeed, given $s<t$ in $[0,T]$, either $X \geq \varepsilon$ on $[s,t]$, or there exist  $s_1 \leq t_1 \in [s,t]$ with $X(s_1),X(t_1) \leq \varepsilon$, with $X \geq \varepsilon$ on $(s,s_1)$ and $(t_1,t)$ (these intervals might be empty if $X \leq \varepsilon$ in $t$ or $s$). Then one has 
\begin{align*}
  \left|X(t) - X(s) \right| 
  &\leq \left|X(t) - X(t_1)\right| + \left|X(t_1)\right| + \left|X(s_1)\right| + \left|X(s)-X(s_1) \right|\\
  &\leq 2 \varepsilon + \omega_\varepsilon(t_1-t) + \omega_\varepsilon(s-s_1).
\end{align*}
It follows that $\hat{X}$ is non-negative, finite and continuous on $[0,T]$. Note that since $\mathcal{S}(V,\xi,x)$ is stable under the maximum operation, one can find an increasing sequence $X^n$ in $\mathcal{S}(V,\xi,x)$ converging to $\hat{X}$ uniformly. One then simply passes to the limit to check that
$$ \hat{X}>0 \mbox{ on }[s,t] \;\; \Rightarrow \;\; \hat{X}(t) = \hat{X}(s) + \int_s^t V(\hat{X}(u)) du + \xi_{s,t}.$$
\end{proof}

\revP{
For any given triplet $(V,\xi,x)$ as above, we will now denote by $\hat{X}(V,\xi,x)$ the maximal element of $\mathcal{S}(V,\xi,x)$ given by the previous proposition.
}

\begin{proposition} \label{prop:rsde}
Let $V$ admit a Lipschitz continuous extension to $[0,\infty)$. \revP{Let $(X,R)$ be the unique continuous solution to
\begin{equation}\begin{aligned} \label{eq:rsde}
  dX(t) &= V(X(t)) dt + d\xi(t) + dR(t),\ X \geq 0,\  dR\geq 0,\ dR(t) 1_{\{X(t)>0\}} = 0,\\
  X(0) &= x, \;\; R(0)=0.
\end{aligned}\end{equation}
Then  $X$ $=$ $\hat{X}(V,\xi,x)$.} In particular, $\xi\mapsto \hat{X}$ is continuous in supremum norm.
\end{proposition}

\begin{proof}
Let ${X}$ solve \eqref{eq:rsde}. Since ${X} \in \mathcal{S}(V,\xi,x)$, 
clearly ${X} \leq \hat{X}$. Then if $\hat{X}>{X}$ on $[s,t]$, clearly $\hat{X}>0$ on this interval, so that
\begin{align*}
  \hat{X}(t)- {X}(t) 
  &=  (\hat{X}(s)- {X}(s)) + \int_s^t (V(\hat{X}(u)) - V({X}(u)) ) du - \int_s^t dR(u)\\ &\leq (\hat{X}(s)- {X}(s)) + \int_s^t C_V \left|\hat{X}(u)-{X}(u)\right| du
\end{align*}
\revP{where $C_V$ is the Lipschitz constant of $V$}, so that by Gronwall's lemma 
  $$\hat X(t)- {X}(t) \leq  (\hat X(s)- {X}(s)) e^{C_V(t-s)}.$$
Letting $s  \downarrow \inf \{r \in [0,t]:\ \hat{X} > {X} \text{ on } [r,t]\}$ we obtain that $\hat{X}(t)- {X}(t) \leq 0$, a contradiction.
\end{proof}

%

\begin{proposition} \label{prop:TKsde}
Let $\xi \in C([0,T])$, $V \in Lip(\R_+)$ \revP{and bounded from above,} with associated flow $\varphi^V$. Let $\{t_i^n\}_{n \geq 0}$ be a sequence of partitions of $[0,T]$ with step size $\pi^n := \sup_{i} |t^n_{i+1} - t^n_i|$ $\to$ $0$ as $n$ $\to$ $\infty$. For $n \geq 0$, define $L^n$ by
\begin{equation}\label{eq:discrete_SDE}\begin{aligned}
  L^n(t^n_{i+1}) &=  \left(\varphi^V(t^n_{i+1}-t^n_{i}, L^n_{t^n_{i}}) + \xi_{t^n_{i}, t^n_{i+1}} \right)_+\\
  L^n(0) &= \ell_0.
\end{aligned}\end{equation}
 Let $(L,R)$ be the (unique continuous) solution to the reflected SDE
\begin{align*}
  dL(t) &= V(L(t)) dt + d\xi(t) + dR(t), \;\;L(t) \geq 0, \; dR(t) \geq 0, \; 1_{\{L(t)>0\}}dR(t)=0\\
  L(0) &= \ell_0,\;\;\;R(0)=0.
\end{align*}
Then, $L^n$ converges uniformly to $L$ on $[0,T]$.
\end{proposition}

\begin{proof} 

\revP{Given $n$, $i$ $\geq 0$, let $k = \sup\{ j \leq i, t^n_j = 0\}$, or $k=0$ if this set is empty. Then one has
\[L^n(t^n_i) \leq L^n(t^n_k) + \|V_+\|_{\infty} (t_i^n - t_k^n) + |\xi_{t_k^n,t_i^n}| \leq \ell_0 + \|V_+\|_{\infty}  T + 2 \|\xi\|_{\infty}.\]
Hence, the $(L^n(t^n_i))$ are uniformly bounded, and since $V$ is continuous we may assume w.l.o.g. that $V$ is bounded.

We then note that there exists a modulus $\tilde{\omega}$ such that for all $n$, for all $t_i^n \leq t^n_j$, one has
\begin{equation} \label{eq:modLn}
\left|L^n(t^n_i) - L^n(t^n_j) \right| \| \leq \tilde{\omega}(t^n_j - t^n_i).
\end{equation}
}
Indeed, taking $t_i^n < t^n_j$, we distinguish two cases :

(1) If $L^n(t^n_k)>0$, for each $i < k < j$, we then have 
$$\left|L^n(t^n_i) - L^n(t^n_j) \right| \| \leq \|V \|_\infty (t^n_j - t^n_i) + \omega(t^n_j - t^n_i),$$ where $\omega$ is the modulus of continuity of $\xi$.

(2) Otherwise considering the first and last times where $L^n = 0$ between $t_i^n$ and  $t^n_j$ and applying the above bound, we obtain $$\left|L^n(t^n_i) - L^n(t^n_j) \right| \leq 2 \left(\| V \|_\infty (t^n_j - t^n_i) + \omega(t^n_j - t^n_i)\right).$$

\revP{
We then extend $L^n$ to all of $[0,T]$ by letting $L^n(0)=\ell_0$ and then
$$L^n(s) = L^n(t^n_i) + \int_{t^n_i}^{s\wedge \rho^n_i} V(L^n(u)) du, \;\; t^n_i \leq s < t^n_{i+1},  \mbox{ where } \rho^n_i = \inf\{s > t_i^n, L^n(s)=0\},$$
$$  L^n(t^n_{i+1}) =  \left( L^n(t^n_{i+1} -) + \xi_{t^n_i,t^n_{i+1}} \right)_+.$$

We then obtain from \eqref{eq:modLn} that for all $t \leq t'$ in $[0,T]$,
\begin{equation*}
\left|L^n(t') - L^n(t)\right| \leq  \varepsilon_n + \tilde{\omega}(t'-t),
\end{equation*}
where $\varepsilon_n = \|V \|_\infty \pi_n + \omega(\pi_n)$ $\to$ $0$ as $n \to \infty$. By an Arzelà-Ascoli argument, this implies that, passing to a subsequence if necessary, $L^n \to \hat{L}$ (locally uniformly) for some continuous $\hat{L}$, and it is enough to show that $\hat{L} = L$.
}

Letting $$R^{n,1}(s) := \sum_{t_{i+1}^n \leq s} \left(L^n(t^n_{i+1} -) + \xi_{t^n_i,t^n_{i+1}} \right)_- , $$
\revP{$$R^{n,2}(s) := (-V(0)) \int_0^s 1_{\{L^n(u) = 0\}} du,$$
note that $R^{n,2}$ is identically $0$ unless $V(0)< 0$, so that $R^n:= R^{n,1}+R^{n,2}$ is nondecreasing.} In addition, one has
$$L^n(t^n_i) = \int_0^{t^n_i} V(L^n(s))ds + \xi_{0, t^n_i} + R^n(t^n_i),$$
and it follows that $R^n$ converges uniformly to some $\hat{R}$, which is nondecreasing and such that
$$\hat{L}(t) = \int_0^t V(\hat{L}(s))ds + \xi_{0,t} + \hat{R}(t).$$
\revP{Note that this implies in particular that $\hat{R}$ is continuous}. It only remains to prove that $\hat{L}(t) d\hat{R}(t) =0$. Assume that $\hat{L}(s) \geq \varepsilon >0$. Then for $n$ large enough, one has $L^n(s) \geq \varepsilon/2$, and then taking $h$ such that for instance $ \|V \|_\infty h + \omega(h) \leq \varepsilon/4$, one has $L^n >0$ on $[s-h, s+h]$. In particular, $dR^n([s-h, s+h]) = 0$, and passing to the limit, $d\hat{R}([s-h,s+h]) =0$, and we have proven that $1_{\{ \hat{L}(t) \geq \varepsilon\}} d\hat{R}(t)=0$, for all $\varepsilon >0$.
\end{proof}

\begin{proposition} \label{prop:compRsde}
Let $V^1, V^2$ be locally Lipschitz on $(0,+\infty)$, bounded from above on $[1,\infty)$, $\xi$ be a continuous path, $x$ $\in$ $\R_+$, and let $\hat{X}^1=\hat{X}(V^1,\xi,x)$, $\hat{X}^2=\hat{X}(V^2,\xi,x)$. Then 
$$V^1 \geq V^2 \mbox{ on } (0,+\infty) \;\Rightarrow \; \hat{X}^1 \geq \hat{X}^2 \mbox{ on } \R_+.$$
\end{proposition}

\begin{proof}
Fix $x \geq \varepsilon >0$, let $V^{1,\varepsilon} = V^1 + \varepsilon$ and $\hat{X}^{1,\varepsilon}$ be the corresponding solution reflected at $\varepsilon$ (i.e. $\hat{X}^{1,\varepsilon}= \hat{X}(x - \varepsilon, V^{1,\varepsilon}(\cdot+\varepsilon),\xi) + \varepsilon$). We first prove that $\hat{X}^{1,\varepsilon} > \hat{X}^{2}$. We proceed by contradiction, and let $t = \inf\{s >0, \hat{X}^{1,\varepsilon}(s) < \hat{X}^{2}(s)\}$. By continuity of $\hat{X}^{1,\varepsilon}, \hat{X}^2$ it holds that for some $\delta>0$, $V^{1,\varepsilon}(\hat{X}^{1,\varepsilon}(s)) > V^2(\hat{X}^2(s))$ for $s \in [t, t+ \delta)$. \revP{Note that $V^{1,\varepsilon}(\cdot+\varepsilon)$ is Lipschitz continuous in a neighbourhood of $0$, so that we can use Proposition \ref{prop:rsde} to obtain}, for $s \in [t, t+ \delta)$,
$$ \hat{X}^{1,\varepsilon}(s) - \hat{X}^{2}(s) = \int_t^s (V^{1,\varepsilon}(\hat{X}^{1,\varepsilon}(u)) - V^2(\hat{X}^2(u))) du + \int_t^s dR^{1,\varepsilon}(u) >0,$$
which is a contradiction.

By the same argument, we see that $\hat{X}^{1,\varepsilon}$ decreases as $\varepsilon \downarrow 0$, and as in the proof of Proposition \ref{prop:Xhat} we can show that the limit $\tilde{X}^1$ is in $\mathcal{S}(V,\xi,x)$. This yields $\hat{X}^2 \leq \tilde{X}^1 \leq \hat{X}^1$ which finishes the proof.
\end{proof}

We next analyze the boundary behavior of the solutions to \eqref{eq:sde}. The first result, Proposition \ref{prop:regular_noise} below, shows that if the signal $\xi$ is too regular compared to the singularity of $V$ at zero, then zero is absorbing or repelling depending on the sign of $V$. In contrast, in the case that $\xi$ is given by Brownian motion, Proposition \ref{prop:reflection_for_BM} below shows that zero may be either absorbing, reflecting or repelling, depending on the singularity of $V$ at zero.

\begin{proposition}\label{prop:regular_noise}
Assume that $\xi \in C^\alpha$, $\alpha \in (0,1]$.Then :

\begin{enumerate}
\item If $V$ is nonincreasing and satisfies $\limsup_{T\to 0} T^{-\alpha} \int_0^T V(s^\alpha) ds = +\infty$, then $$\forall t>0, \hat{X}(t)>0.$$
\item If $V$ is nondecreasing and satisfies $\limsup_{T\to 0} T^{-\alpha} \int_0^T V(s^\alpha) ds = -\infty$, then $$\hat{X}(t)=0 \Rightarrow \forall s \geq t, \hat{X}(s)=0.$$
\end{enumerate} 
 \end{proposition}
 
 \begin{proof}
 (1) The case where $X(0)>0$ is treated in \cite[Prop. 2.2]{M15}, and we only need to prove the case where $X(0)=0$.
 
We fix $\delta>0$, and take $V^\delta\leq  V$ with $V^\delta$ bounded and Lipschitz on $\R_+$, and such that 
 \begin{equation}\label{eq:V'}
 V^\delta(0^+) > \inf_{\delta \geq t\geq s \geq 0} \frac{\xi_{s,t}}{(t-s)}.
 \end{equation}
 Let $X^\delta:= \hat{X}(V^\delta,\xi,x)$. Then by Proposition \ref{prop:compRsde} one has $\hat{X} \geq X^\delta$, and by Proposition \ref{prop:rsde}, for all $s \leq t$,
 $$X^\delta(t) \geq X^\delta(s) + \int_s^t V^\delta(X^\delta(s)) ds + \xi_{s,t}.$$
By \eqref{eq:V'}, $X^\delta$ is not identically $0$ on $[0,\delta]$, and neither is $\hat{X}$. Hence there is a sequence $t_\delta \to 0$ with $\hat{X}_{t_\delta} >0$, and by the case $\hat{X}_0>0$  we conclude that $\hat{X} >0$ on $(0,\infty)$. 

(2) is a consequence of (1) by time-reversal: If for some $s \leq t$, one has $\hat{X}(s)=0$ and $\hat{X} >0$ on $(s,t)$, then letting $Y(u) = \hat{X}(t-u)$, $Y$ satisfies the assumptions of (1) (with $V$ replaced by $-V$, $\xi$ by $\xi_{t-\cdot}$), and $Y(t-s) = 0$ which is a contradiction.
 \end{proof}
 
When $\xi$ is a standard Brownian motion, one has a complete classification of the boundary behavior at $0$.


\begin{proposition}\label{prop:reflection_for_BM}
Let $V$ be locally Lipschitz on $(0,+\infty)$, bounded from above on $[1,\infty)$, $x \in \R_+$, $B$ be a linear Brownian motion, and let $\hat{X} = \hat{X}(V,B,x)$. Define
$$I^+ =  \int_0^1\int_x^1 e^{ 2\int_x^y V(u) du} dy dx,  \;\;\;I^- = \int_0^1\int_x^1 e^{- 2\int_x^y V(u) du} dy dx.$$
Then one has the following four possible cases :
\begin{enumerate}
\item (Regular boundary) If $I^+< \infty, I^-<\infty $, then :
$$\forall t>0, \P(\hat{X}(t) =0)=0, \;\;\; \P(\exists s \leq t, \hat{X}(s) = 0) >0.$$
\item (Exit boundary) If $I^- = \infty, I^+ < \infty$ :
$$ \P(\exists s \leq t, \hat{X}(s) = 0) >0, \;\;\; \P(\exists s<t, \hat{X}(s)=0, \hat{X}(t)>0)=0.$$
\item (Entrance boundary)  If $I^+ = \infty, I^- < \infty$ :
$$\P(\forall t>0, \hat{X}(t)>0) =1, $$
\item (Natural boundary) If $I^+=I^-=\infty$ :
$$\mbox{ If $x > 0$, then } \P(\forall t>0, \hat{X}(t)>0) =1, \mbox{ if $x=0$ then } \P(\forall t, \hat{X}(t) = 0) = 1.$$
\end{enumerate}
\end{proposition}

\begin{proof}
This is mostly standard (cf.\ e.g.\ \cite[sec. 15.6]{KT81}), noting that $I^+ = \int_0^1 dm(x) \int_x^1 ds(y)$, $I^- = \int_0^1 ds(x) \int_x^1 dm(y)$ where $s$ is the scale function and $m$ is the speed measure associated to \eqref{eq:sde}.

In case (1) the diffusion admits several possible boundary behaviors (so that $\mathcal{S}(V,\xi,x)$ is in general infinite), but it is known that there exists a process $X \in \mathcal{S}(V,\xi,x)$ which is instantaneously reflected i.e. such that $\P(X(t) = 0)=0$ for all $t>0$. Since $\hat{X} \geq X$ this implies that $\P(\hat{X}(t) = 0)=0$.
\end{proof}

\subsection{A Trotter-Kato formula}\label{sec:trotter}

In this section we establish a Trotter-Kato formula for viscosity solutions to \eqref{eq:main-SPDE}. 

From Theorem \ref{thm:app_wp} recall that for $u_0 \in BUC(\R^N)$, $\xi,\z \in C([0,T];\R)$ we have 
\begin{equation} \label{eqn:EstLS}
 \left\|S^\xi(u_0) - S^\zeta(u_0)\right\|_\infty \leq \Phi \left( \left\|\xi_{0,\cdot} - \zeta_{0,\cdot}\right\|_\infty\right),
 \end{equation}
for some function $\Phi$ as in Theorem \ref{thm:app_wp}.

We now show that, as a consequence of this estimate, it is possible to define $S^\xi(u_0)$ for paths $\xi$ admitting jumps, in such a way that the estimate \eqref{eqn:EstLS} remains true.

To this end, let $\xi$ be a piecewise continuous path on $[0,T]$ with jumps $\D\xi(t_i):=\xi(t_i+)-\xi(t_i-)$ for $i=1,\dots,m-1$ along a partition $(t_i)_{0\leq i \leq m}$ of $[0,T]$. We then define $u = S^\xi(u_0)$ as the solution to
\begin{align*}
u(0,\cdot) &= u_0(\cdot) , \;\;\; \\
u(t) &= \left(S^{\xi_{|[t_i,t_{i+1}]}} u(t_{i})\right)(t) \mbox{ on } [t_i, t_{i+1}), \forall 0 \leq i \leq m-1,\\
u(t_{i+1}) &= S_{H}(\Delta \xi(t_{i+1})) (u(t_{i+1}-)), \;\;0 \leq i \leq m-2.
\end{align*}

\revP{
This definition is in the spirit of Marcus' canonical solutions to SDE driven by jump processes \cite{Marcus81}, and consists in replacing each jump $\Delta \xi$ by a "fictitious time" during which the equation $\partial_t +H(Du)=0$ is solved. This interpretation is actually used in the proof of the following proposition.
}

\begin{proposition}\label{prop:ext_jumps}
Let $u_0 \in BUC(\R^N)$ and $\xi$, $\zeta$ be piecewise-continuous paths. Then, \eqref{eqn:EstLS} holds.
\end{proposition}
\begin{proof}
The idea is to change the parametrization of $\xi$, $\zeta$ in order to replace the piecewise-continuous paths by continuous paths. 

We replace $[0,T]$ by $[0,\tilde{T}]$, obtained from $[0,T]$ by adding an interval for each jump of $\xi$ and $\zeta$. For instance, say that $\xi$ and $\zeta$ have jumps at the points $(t_i)_{i=1,\ldots,m-1}$. We then \revP{fix $\delta>0$ , let $\tilde{T} = T + 2(m-1)\delta$, and let 
$$I = \cup_{i=1}^{m-1} [t_i + (2i-1) \delta, t_{i} + 2i\delta),\quad J =[0,\tilde{T}]\setminus I.$$
We further fix a continuous function $\psi^\delta$ satisfying
\[ 0 \leq \psi^\delta \leq 1, \]
\[ \psi^\delta = 0 \mbox{ on } I, \;\; \psi^\delta>0 \mbox{ on the interior of } J,\]
\[ \int_0^{t_i + (2i-1)\delta} \psi^\delta(v) dv = t_i,\;\; \forall i \in \{1,\ldots,m\}.\]

Then $s^\delta(t) := \int_0^t \psi^\delta(u)du$ defines a bijection from $J$ to $[0,T]$.

We define $\tilde{\xi}$ such that $\tilde{\xi} = \xi \circ s^\delta$ on $J$, $\tilde{\xi}$ is continuous on $[0,\tilde{T}]$, and $\tilde{\xi}$ is affine linear on each interval of $I$ and analogously for $\tilde{\zeta}$.
We further let
$$\tilde{F}^\delta(t,\cdot) = F(s^\delta(t), \cdot)\psi^\delta(t),\quad t \in [0,\tilde{T}].$$

Let $\td u^{\td\xi}$ be the solution to
\begin{align*}
  d \td u &= \tilde{F}(t,x,\td u,D\td u,D^2\td u)dt - H(D\td u) \circ d{\tilde{\xi}}(t) \\
  \td u(0)&=u_0, 
\end{align*}
and define $\td u^{\td\zeta}$ analogously. Then 
  $$S^\xi(u_0)(t,\cdot) = \td u^{\td\xi}((s^{\delta})^{-1}(t),\cdot),\quad S^\zeta(u_0)(t,\cdot) = \td u^{\td \zeta}((s^{\delta})^{-1}(t),\cdot),$$
so that
  $$ \left\|S^\xi(u_0) - S^\zeta(u_0)\right\|_\infty \leq  \left\|\td u^{\td\xi} - \td u^{\td\zeta}\right\|_\infty 
  \leq \tilde{\Phi} \left( \left\|\tilde{\xi}_{0,\cdot} - \tilde{\zeta}_{0,\cdot}\right\|_\infty\right)=\tilde{\Phi} \left( \left\|\xi_{0,\cdot} - \zeta_{0,\cdot}\right\|_\infty\right),$$
  where $\tilde{\Phi}$ is given by Theorem \ref{thm:app_wp} applied to $\tilde{F}, \tilde{T}$. Now since $\tilde{F}$ satisfies Assumption \ref{asn:F} (2)-(3)-(5) with the same quantities as $F$, and since $\tilde{T}$ may be taken as close to $T$ as one wishes \revP{by letting $\delta \to 0$}, it follows that the estimate above also holds with $\tilde{\Phi}$ replaced by $\Phi$.
}

\end{proof}

\begin{corollary}[Trotter-Kato formula]\label{cor:trotter-kato}
Let $\xi \in C([0,T])$, $u_0 \in BUC(\R^N)$ and let $u$ be the corresponding viscosity solution to \eqref{eq:main-SPDE}. Further let $(t_i^n)$ be a sequence of partitions of $[0,T]$ with step-size going to $0$. Define $u^n$ by
$$u^n(t,\cdot) := \left(S_F(t_j^n,t) \circ S_{H}(\xi_{t_{j-1}^n,t_j^n}) \circ  S_F(t_{j-1}^n,t_j^n ) \circ \cdots \circ S_{H}(\xi_{0,t^n_1}) \circ S_F(0,t_1^n)\right) (u_0),$$
for $t \in [t_j^n, t_{j+1}^n)$. Then 
  $$\|u^n - u\|_{C([0,T] \times \R^N)}\to0\quad\text{for }n\to\infty.$$
\end{corollary}
\begin{proof}
We have $u^n = S^{\xi^n}(u_0)$, where $\xi^n$ is the piecewise constant path equal to $\xi_{t_i^n}$ on $[t_i^n, t_{i+1}^n)$. The claim now follows from Proposition \ref{prop:ext_jumps}.
\end{proof}

\subsection{Proof of Theorem \ref{thm:main}}\label{sec:main_thm_proof}
Let $t_i^n = \frac{ti}{n}$ and 
  $$ u^n(t) := S_{H}(\xi_{{t_{n-1}^n,t_n^n}}) \circ S_F(\frac{t}{n}) \circ \cdots \circ S_{H}(\xi_{t_0^n,t_1^n}) \circ S_F(\frac{t}{n}) u^0.$$
By Corollary \ref{cor:trotter-kato}, one has
  $$ u(t,\cdot) = \lim_{n \to \infty} u^n(t, \cdot).$$
Proposition \ref{prop:InfSup} combined with Assumption \ref{asn:odeF} implies 
$$D^2u^n(t,\cdot) \leq \frac{Id}{L^n(t)},$$
where $L^n$ is defined by the induction 
$$L^n(0) = \ell_0, \;\; L^n(t^n_i) = \left( \vp^{V_F}(\frac{t}{n})(L^n(t^n_{i-1})) - \xi_{t_{i+1}^n,t_{i}^n}\right)_+.$$
Now If $V_F$ admits a Lipschitz extension to $[0,\infty)$, then as $n \to \infty$  $L^n$ converges to $L$ by Proposition \ref{prop:TKsde} and we are done.

Let now $V$ be only locally Lipschitz continuous. First assume that $L>\varepsilon>0$ on $[0,t]$ for some $\ve>0$. Let $\tilde{V}$ be Lipschitz continuous on $[0,\infty)$ with $\tilde V=V$ on $(\ve,+\infty)$ and let $\tilde{L}$, $\tilde{L}^n$ be the solutions to \eqref{eq:reflected}, \eqref{eq:discrete_SDE} with $V$ replaced by $\tilde V$ respectively. Then $L=\tilde{L}$ and $\tilde{L}=\lim_n \tilde{L}^n$ by Proposition \ref{prop:TKsde}. Thus, $\tilde{L}^n > \ve$ for $n$ large enough, which implies $L^n = \tilde{L}^n$ and $\lim_n L^n = L$.

Now assume that $L(s)=0$ for some $s\in[0,t]$ and $L(t)>0$ (otherwise there is nothing to prove). Hence, for all $\varepsilon>0$ (small enough), there exists an $s_\varepsilon$ $\in$ $(0,t)$ with $L_{s_\varepsilon} = \varepsilon$, and $L\geq \varepsilon$ on $[s_\varepsilon,t]$. Let now $u^\varepsilon$ be the solution to \eqref{eq:main-SPDE} on $(s_\varepsilon,t]\times \R^n$ with $u^\varepsilon(s_\varepsilon,\cdot)=S_H(-\varepsilon)u(s_\varepsilon,\cdot)$. By Proposition \ref{prop:InfSup}, $D^2u^\varepsilon(s_\varepsilon,\cdot) \leq \varepsilon Id$, and since $L >0$ on $[s_\varepsilon,t)$, we may apply the Trotter-Kato formula as in the previous case to conclude that $D^2 u^\varepsilon(t,\cdot) \leq \frac{Id}{L(t)}$. Finally, note that $u^\varepsilon(t)$ is the solution to \eqref{eq:main-SPDE} driven by $\xi^\varepsilon = \xi + \varepsilon 1_{[s_\varepsilon,t]}$. Since $\xi^\ve \to \xi$ uniformly as $\varepsilon \to 0$, we conclude the proof by Proposition \ref{prop:ext_jumps}. 

\section{Semiconvexity preservation}\label{sec:semiconv}

In this section we provide sufficient conditions on $F$ to satisfy Assumption \ref{asn:odeF}. From \cite{LM06} we recall

\begin{proposition} \label{thm:LM}
Let $F = F(t,x,p,A) \in C([0,T]\times\R^N\times\R^N\times S^N)$ be degenerate elliptic and such that, for all $t \geq 0$, $x, p \in \R^N$, $q \neq 0 \in \R^N$,
\begin{align} \label{eq:LMcond}
  (y,A) \mapsto F(t,x + y,p , B) \text{ is convex on } (\R q)^{\perp} \times X_{q}, 
\end{align}
where $X_q = \left\{A \in S^N, A q = 0, A >0 \mbox{ on } (\R q)^{\perp}\right\} $, $B q = 0 , \; B = A^{-1} \mbox{ on }  (\R q)^{\perp}$.

Let $u$ be coercive in $x$ i.e.
$$ \lim_{|x|\to \infty} \inf_{t \in [0,T]} \frac{u(t,x)}{|x|} = +\infty$$
and a classical supersolution on $[0,T] \times \R^N$ to 
\begin{equation} \label{eq:F}
\partial_t u = F(t,x,Du,D^2 u),
\end{equation}
 and let 
$$u_{**}(t,x) := \inf \left\{ \sum_{i=1}^m \lambda_i u(t,x_i), \; \; 0\leq \lambda_i\leq 1, \sum_{i=1}^m \lambda_i =1,  \sum_{i=1}^m \lambda_i x_i =x\right\}$$ be the partial convex envelope of $u$. Then $u_{**}$ is a viscosity supersolution to \eqref{eq:F}.
\end{proposition}

\begin{proof}
For the reader's convenience we provide a proof. First note that by continuity of $F$, it is straightforward to see that the assumption \eqref{eq:LMcond} is equivalent to the fact that for any subspace $V \subset \R^n$ which is not reduced to $\{0\}$, the map
\begin{align} \label{eq:LMcondV}
  (y,A) \mapsto F(t,x + y,p , B) \text{ is convex on } V^{\perp} \times X_V, 
\end{align}
where $X_V = \left\{A \in S^N, A_{|V} = 0, A >0 \mbox{ on } V^{\perp}\right\} $, $B_{|V} = 0 , \; B = A^{-1} \mbox{ on }  V^{\perp}$.

Now consider $(t,x) \in (0,T]\times \R^n$, and let $(q,p,A)$ be in the parabolic subjet of $u_{**}$ at $(t,x)$ (we refer e.g. to \cite{CIL92} for definitions). Assume that $u_{**}(t,x)<u(t,x)$ (otherwise there is nothing to prove), let $\lambda_i$, $x_i$, $i=1,\ldots, m$ be such that $u_{**}(t,x)=\lambda_1 u(t,x_1) +\ldots + \lambda_m u(t,x_m)$, and let $V$ be the span of $(x_1-x,\ldots,x_m-x)$.  Then by similar computations as in \cite[pp.272-273]{ALL97}, letting $A_i = D^2 u(t,x_i)$, it holds that
\begin{equation*}
A_i \geq 0, \;\;\; A \leq \left( \sum \lambda_i A_i^{-1}\right)^{-1},
\end{equation*}
\begin{equation*}
q = \sum_{i=1}^m \lambda_i \partial_t u(t,x_i),
\end{equation*}
\begin{equation*}
p=Du(t,x_i), \;\;i=1,\ldots,m.
\end{equation*}

Note that since $u_{**}(t,\cdot)$ is affine in the directions spanned by $V$ in a neighborhood of $x$, one has $A\leq 0$ on $V$, so that by ellipticity 
$$q - F(t,x,p,A) \geq q - F(t,x,p,B),$$
where $B=\left( \sum \lambda_i A_i^{-1}\right)^{-1}$ on $V^\perp$, $B=0$ on $V$, 
and by \eqref{eq:LMcondV}, we obtain
$$q-F(t,x,p,A) \geq \sum_{i=1}^m \lambda_i (\partial_t u(t,x_i) - F(t,x_i,Du(t,x_i), \tilde{A}_i))$$
where $\tilde{A}_i = A_i$ on $V^\perp$, $\tilde{A}_i=0$ on $V$, so that $\tilde{A}_i\leq A_i$, and by ellipticity of $F$ and the fact that $u$ is a supersolution to the equation we finally obtain
$$q-F(t,x,p,A) \geq 0.$$
\end{proof}

We deduce the following 
\begin{theorem}\label{thm:semi-conv}
Let $F = F(t,x,p,A)\in C([0,T]\times\R^N\times\R^N\times S^N)$ be degenerate elliptic 
such that there exists a $\Phi \in \Lip_{loc}(\R_+;\R)$ with $\Phi(0+)\ge 0$ such that for all $\lambda \in \R_+$,  $t\in [0,T], x, p \in \R^N, q \neq 0 \in \R^N$,
\begin{equation} \label{eq:LM2cond} \left\{\begin{aligned}
  &(y,A) \mapsto  F(t,x + y,p - \lambda (x+y) , B - \lambda I) + \frac{1}{2} \Phi(\lambda) \left|x+y\right|^2 \\
  & \mbox{ is convex on }(\R q)^\perp \times X_{q}, \\
\end{aligned}\right.\end{equation}
where $X_q = \left\{A \in S^N, Aq = 0, A >0 \mbox{ on } (\R q)^{\perp} \right\}$, $B q = 0 , \; B = A^{-1} \mbox{ on }  (\R q)^{\perp}$. 
Let $u_0 \in C^2(\R^N)$ satisfy $D^2 u_0 \geq -\lambda_0 I$ for some $\l_0\ge 0$ and assume that $u$ satisfies for some $K>0$,
\begin{equation}\label{eq:u_growth}
   |u(t,x)| \le K(1+|x|)\quad\forall x\in\R^N,\ t\in [0,T]
\end{equation}
and is a classical solution to
\begin{equation} \label{eq:F2}
\left\{\begin{array}{l}	
\partial_t u  = F(t,x,Du,D^2u) , \\u(0,\cdot) = u_0,\end{array}\right.
\end{equation} 
 then if $\lambda(t)$ is the solution to
 \begin{equation}\label{eq:lambda}
 \left\{\begin{array}{l}	\dot{\lambda}(t) = \Phi(\lambda(t)), \\ \lambda(0) = \lambda_0,\end{array}\right.
\end{equation} 
  one has $D^2 u(t,\cdot) \geq -\lambda(t) I$ for all $t  \geq 0$.
\end{theorem}
\begin{proof}
Let $\ve>0$ arbitrary, fix and let $\l^\ve$ be the solution to \eqref{eq:lambda} with initial condition $\l^\ve(0)=\l_0+\ve$. We set $v(t,x) := u(t,x) + \frac{1}{2} \lambda^\ve(t)|x|^2$. Since $\l(t) > 0$, $v(t)$ is coercive, in the sense that $\inf_{t\in[0,T]}\frac{v(t,x)}{|x|}\to \infty$ for $|x|\to\infty$. Moreover, $v$ is a classical solution to 
\begin{equation}\begin{aligned}\label{eq:tdF}
  \partial_t v 
  &= F(t,x,Du,D^2u)+\frac{1}{2}\Phi(\l^\ve(t))|x|^2\\
  &= F(t,x,Dv-\l^\ve(t)x,D^2v-\l^\ve(t)Id)+\frac{1}{2}\Phi(\l^\ve(t))|x|^2\\
  &=: \td F(t,x,Dv,D^2v).
\end{aligned}\end{equation}
By \eqref{eq:LM2cond}, $\td F$ satisfies \eqref{eq:LMcond}. Hence, by Proposition \ref{thm:LM}, the convex envelope $v_{**}$ 0f $v$ is a supersolution to \eqref{eq:tdF}. Equivalently, $\hat u := v_{**} - \frac{1}{2} \lambda^\ve(t)|x|^2$ is a supersolution to \eqref{eq:F2}. By \eqref{eq:u_growth} we have that
  $$v(t,x)\ge \frac{1}{2}\l^\ve(t)|x|^2-K-K|x|$$
for all $x\in \R^d$ which implies that
  $$v_{**}(t,x) \ge \frac{1}{2}\l^\ve(t)|x|^2-\td K-K|x|,$$
for some $\td K>0$ and all $x\in \R^d$. Hence, $\hat u \ge -\td K(1+|x|)$ and we may apply the comparison result \cite[Theorem 4.2]{GGIS91} to obtain
  $$u \le \hat u.$$
On the other hand, since $v_{**}\le v$ we have that
\begin{align*}
  \hat u 
  &\le v - \frac{1}{2} \l^\ve(t)|x|^2 = u.
\end{align*}
Hence, $\hat u = u$ and, since $v_{**}$ is convex, we conclude
  $$D^2 u = D^2 \hat u = D^2 v_{**} - \l^\ve(t) Id \ge - \l^\ve(t) Id.$$
Since this is true for all $\ve>0$ the proof is finished.
\end{proof}

\revP{
Since Theorem \ref{thm:semi-conv} applies to classical solutions only, in order to obtain results for general viscosity solutions we must proceed by suitable approximations. The following corollary is an immediate consequence of the stability of viscosity solutions \cite{CIL92}.

\begin{corollary} \label{lem:StabSemiconv}
Let $F_\varepsilon$ satisfy the assumptions of Theorem \ref{thm:semi-conv} for a given $\Phi_\varepsilon$, and let $u_\varepsilon$ be classical solutions to
\begin{equation*} 
\left\{\begin{array}{l}	
\partial_t u_\varepsilon  = F_\varepsilon(t,x,Du_\varepsilon,D^2u_\varepsilon) , \\u(0,\cdot) = u_0^\varepsilon,\end{array}\right.
\end{equation*} 
with $D^2 u^\varepsilon_0 \geq - \lambda_0^\varepsilon Id$. Further assume that $(F_\varepsilon,u_0^\varepsilon,\Phi_\varepsilon,\lambda_0^\varepsilon)$ converges locally uniformly to $(F,u_0,\Phi,\lambda)$, with $F$ satisfying Assumption \ref{asn:F}, $u_0$ $\in$ $BUC(\R^N)$, and $\Phi \in \Lip_{loc}(\R_+;\R)$. Then, letting $u$ be the unique bounded viscosity solution to
\begin{equation*}
\left\{\begin{array}{l}	
\partial_t u  = F(t,x,Du,D^2u) , \\u(0,\cdot) = u_0,\end{array}\right.
\end{equation*}
  one has $D^2 u(t,\cdot) \geq -\lambda(t) Id$ for all $t  \geq 0$ where $\lambda(t)$ is the solution to
 \begin{equation*}
 \left\{\begin{array}{l}	\dot{\lambda}(t) = \Phi(\lambda(t)), \\ \lambda(0) = \lambda_0.\end{array}\right.
\end{equation*}
\end{corollary}
}

We now give examples (corresponding to the cases in Proposition \ref{ex:LMcond1}) for which \eqref{eq:LM2cond} holds.

\begin{proposition}\label{ex:LMcond}
\begin{enumerate}
\item Let
    $$F=F(t,x,p)  \in C([0,T];C^2_b(\R^N \times \R^N)).$$
 Then \eqref{eq:LM2cond} is satisfied with
    $$\Phi(\lambda) =  \left\| F_{xx} \right\|_\infty +  2 |\lambda |\left\| F_{xp} \right\|_\infty + \lambda^2  \left\| F_{pp} \right\|_\infty.$$
 More generally, let $F=F(t,x,p)  \in C([0,T]\times \R^N \times \R^N)$ such that $(x,p) \mapsto F(t,x,p)$ is semiconvex of order $C_F$. Then, \eqref{eq:LM2cond} is satisfied with 
        $$\Phi(\lambda) =  C_F(1+\l^2).$$
\item Let
             $$F(x,p,A) = Tr(a(x,p) A) \in C(\R^N \times\R^N \times S^N),$$
            where $a(x,p) \in C^2(\R^N\times \R^N)$ is nonnegative, has bounded second derivative 
            and $(y,p)\mapsto \sqrt{a(y,p)}$ is convex. Then \eqref{eq:LM2cond} is satisfied with
             $$\Phi(\lambda) =N \lambda \left\| a_{xx} \right\|_\infty +  2N \lambda^2 \left\| a_{xp} \right\|_\infty + N\lambda^3  \left\| a_{pp} \right\|_\infty.$$
\item  Let 
    $$F = F(t,A) \in C([0,T] \times S^N)$$
  be convex and non-decreasing in $A\in S^N$. Then \eqref{eq:LM2cond} is satisfied with $\Phi = 0$.  
\item  Let $F=F(t,x,p,A)\in C([0,T]\times\R \times \R \times \R)$ 
such that $(x,p) \mapsto F(t,x,p,A)$ is semiconvex of order $C_F(A)$. Then, \eqref{eq:LM2cond} is satisfied with
     $$\Phi(\lambda) = C_F(\l)(1+\l^2).$$
\end{enumerate}
\end{proposition}
\begin{proof}
(1): Immediate.
  
(2): For $\l\in\R_+$, $x,p\in\R^N$, $q\ne0 \in\R^N$ we aim to prove convexity of 
\begin{align*}
   (y,A)
   \mapsto& F(x+y,p-\l(x+y),B-\l I)+\frac{1}{2}\Phi(\l)|x+y|^2 \\
   =& a(x+y,p-\l(x+y))Tr(B)\\&-a(x+y,p-\l (x+y))\l N +\frac{1}{2}\Phi(\l)|x+y|^2\\
   =&:F_1(x+y,p,B)+F_2(x+y,p).
\end{align*}
For the first part, $F_1$, we note that, by \cite[Theorem 3.1, Remark (ii)]{LM06}, convexity of $(y,A)\mapsto F_1(x+y,p,B)$ follows from convexity of $\sqrt{a}$. For the second part $F_2$ we note that 
\begin{align*}
   D_{yy}F_2 
   &= -\l N D_{yy}a(x+y,p-\l (x+y)) + N\l^2 D_{yp}a(x+y,p-\l (x+y)) \\
   &+ N\l^3 D_{pp} a(x+y,p-\l (x+y)) + \Phi(\l)\\
   &\ge -\l N \|D_{yy}a\|_\infty - N\l^2 \|D_{yp}a\|_\infty - N\l^3 \|D_{pp} a\|_\infty + \Phi(\l)\\
   &\ge0.
\end{align*}

(3): Let $q \ne 0 \in\R^N$. By \cite[Appendix]{ALL97} the map $A \mapsto A^{-1}$ is convex on $X_q$, which implies \eqref{eq:LM2cond} with $\Phi = 0$.

(4): Note that we have $X_q = \{0\}$ in \eqref{eq:LM2cond} and thus only convexity in $y$ has to be checked, which easily follows from semiconvexity of $F$.
\end{proof}

\revP{
We are finally in the position to prove Proposition \ref{ex:LMcond1}.

\begin{proof}[Proof of Proposition \ref{ex:LMcond1}]
Note that Assumption \ref{asn:odeF} deals with semiconcavity bounds whereas Theorem \ref{thm:semi-conv} yields semiconvexity bounds, so in each case we pass from one to the other by considering $\tilde{u} = -u$, $\tilde{F}:= -F(t,x,-r,-p,-X)$. We also make the change of variables $\ell = \lambda^{-1}$ so that to a given $\Phi$ corresponds $V_F(\ell)=-\ell^2 \Phi(\ell^{-1})$.

All the cases then follow by combining Corollary \ref{lem:StabSemiconv} and Proposition \ref{ex:LMcond}. The only point to be verified is the existence of approximations by classical solutions.

We present the details for the case (1): Let $v$ be the viscosity solution to
\[\partial_t v = F(t,x,Dv), \;\; v(0,\cdot) = v_0. \]
For $\ve>0$ let $w^\ve$ be the classical solution (cf.\ e.g.\ \cite[chapter XIV]{L96})) to
  \begin{equation*}\label{eq:plp-approx}\begin{aligned}
    \partial_t w^\ve &= -F(t,x,-Dw^\ve) + \ve \D w^\ve,\\
     w^\ve(0) &= - u_0^\ve,
  \end{aligned}\end{equation*}
  where $u_0^\ve \in C^2_b(\R^N)$ converges to $u_0$ locally uniformly. Note that if $F_1$, $F_2$ satisfy \eqref{eq:LM2cond} with $\Phi_1,\Phi_2$, then so does $F_1+F_2$ with $\Phi_1+\Phi_2$. Hence by Proposition \ref{ex:LMcond} (1) and (3), we see that $F_\varepsilon(t,x,p,A)= -F(t,x,-p) + \varepsilon Tr(A)$ satisfies \eqref{eq:LM2cond} with $$\Phi(\lambda) =  \left\| F_{xx} \right\|_\infty +  2 |\lambda |\left\| F_{xp} \right\|_\infty + \lambda^2  \left\| F_{pp} \right\|_\infty.$$
Hence, we can apply Corollary \ref{lem:StabSemiconv} to obtain that
\[D^2 v(t) \leq \lambda(t) Id,\]
where $\dot{\lambda}(t) = \Phi(\lambda(t))$ and $\lambda(0) = \|(D^2 v_0)_+\|_{\infty}$. Noting $\ell(t) = \lambda(t)^{-1}$, one has $\dot{\ell}(t) = V_F(\ell(t))$ with
$$V_F(\ell) =  - \left\| F_{xx} \right\|_\infty \ell^2  -  2\left\| F_{xp} \right\|_\infty \ell - \left\| F_{pp} \right\|_\infty,$$
so that Assumption \ref{asn:odeF} is indeed satisfied.

The cases (2), (3), (4) follow similarly (the existence of smooth solutions for the approximating equations follows for instance from the existence results in \cite[chapter XIV]{L96}).
\end{proof}
}

\revP{
We also need the following standard lemma, we include its proof for completeness.
\begin{lemma} \label{lem:Lipvisc}
Let $F$ be continuous and degenerate elliptic, and given $v_0$ bounded and Lipschitz on $\R^N$ let $v$ solve in viscosity sense
\[\partial_t v = F(Dv,D^2v), \;\; v(0,\cdot)=v_0.\]
Then for all $t \geq 0$, 
\[ \left( \sup v(t,\cdot) - \inf v(t,\cdot)\right) \leq \sup v_0 - \inf v_0,\]
\[\left\|Dv(t,\cdot)\right\|_\infty \leq \left\|Dv_0\right\|_\infty.\]
\end{lemma}
\begin{proof}
The first claim follows by comparing $v$ with solutions of the form $M +t F(0,0)$, taking $M$ equal to $\sup v_0$  and $\inf v_0$.

For a given $z \in \R^N$, note that $v(\cdot,\cdot+z)$ solves the same equation as $v$ with initial condition $v_0(\cdot+z)$, so that by viscosity comparison, for all $t \geq 0$,
\[ \sup_{x \in \R^N} \left( v(t,x+z) - v(t,x) \right) \leq \sup_{x \in \R^N} \left( v_0(x+z) - v_0(x) \right) \leq \|Dv_0\|_\infty |z|.
\]
\end{proof}
}

\section{Optimality}\label{sec:optimality}

In this section we prove the optimality of the estimates given in Theorem \ref{thm:plp} and thereby also the ones given in Theorem \ref{thm:main} by providing an example of an SPDE and suitable initial conditions for which these estimates are shown to be sharp.

We consider the class of functions
\begin{align*}
\mathcal{U}=& \Big\{u\in BUC(\R) \mbox{ is $2$-periodic with } u(x)=u(-x), u(1+x)=u(1-x),\  \forall x \in \R \\
& \;\;\;\; \mbox{ and s.t. }0\leq  u_x \leq 1, u_{xxx} \leq 0 \mbox{ in the sense of distributions on } (0,1)  \Big\}.
\end{align*}
Note that if $u \in \mathcal{U}$, then
\begin{equation} \label{eq:uxx0}\begin{aligned}
 \left\| (u_{xx})_+\right\|_\infty &= u_{xx}(0) = \sup_{\delta \in (0,1)} \frac{u(\delta) - u(0)}{\delta^2}\\
  \left\| (u_{xx})_-\right\|_\infty &= -u_{xx}(1) = -\sup_{\delta \in (0,1)} \frac{u(1) - u(1-\delta)}{\delta^2},
\end{aligned} \end{equation} 
where both of them may take the value $+\infty$.

\begin{theorem} \label{thm:opt}
Let $u^0$ $\in$ $\mathcal{U}$, $\xi$ $\in$ $C_0([0,T])$ and let $u$ be the solution to
\begin{equation} \label{eq:exopt}
  d u +  \frac{1}{2}|u_x|^2 \circ d{\xi}(t) = \frac{1}{4}|u_x|^2 u_{xx} dt , \;\;\; u(0,\cdot) = u^0.
\end{equation}
Then, $u(t,\cdot) \in \mathcal{U}$ for all $t \geq 0$, and 
  $$u_{xx}(t,0)=\frac{1}{L^+(t)},\quad u_{xx}(t,1)=-\frac{1}{L^-(t)},$$
where $L^+$, $L^-$ are the maximal continuous solutions to
\begin{align}
  dL^+(t) &= -\frac{1}{2 L^+(t)}dt + d\xi(t) \mbox{ on } \{L^+>0\}, \;\;L^+(t) \geq 0, \;\; L^+(0) = \frac{1}{\left\| (u^0_{xx})_+\right\|_\infty}, \label{eq:opt_SDE} \\
  dL^-(t) &= -\frac{1}{2 L^-(t)}dt - d\xi(t) \mbox{ on } \{L^->0\}, \;\;L^-(t) \geq 0, \;\; L^-(0) = \frac{1}{\left\| (u^0_{xx})_+\right\|_\infty}. \label{eq:opt_SDE2}
\end{align} 
\end{theorem}

An application of Proposition \ref{prop:reflection_for_BM} yields
\begin{corollary}\label{cor:plp-optimal}
In Theorem \ref{thm:opt} let $\xi = \sigma B$ where $B$ is a Brownian motion. Then
\begin{enumerate}
\item If $\sigma  \leq  1$: a.s. there exists a $T^*$ such that $\|D^2u(t,\cdot)\|_\infty=+\infty$ for all $ t> T^*$.
\item If $\sigma  >  1$: for each $t>0$, a.s. $\|D^2u(t,\cdot)\|_\infty<+\infty$.
\end{enumerate}
\end{corollary}

We next proceed to the proof of Theorem \ref{thm:opt}. We shall concentrate on proving $u_{xx}(t,0)=\frac{1}{L^+(t)}$, the other equality can be obtained analogously. By Theorem \ref{thm:plp} we already know that $L^+(t) \le \frac{1}{u_{xx}(t,0)}$. Since also $L^+$ is the maximal solution to \eqref{eq:opt_SDE}, it only remains to prove that $t\mapsto \frac{1}{u_{xx}(t,0)}$ $\in$ $\mathcal{S}(V,\frac{1}{u^0_{xx}(t,0)},\xi)$, which is a consequence of Proposition \ref{prop:opt1} and Proposition \ref{prop:opt2} below.

\begin{lemma} \label{lem:2}
Let $u^0$ $\in$ $C^6_b \cap \mathcal{U}$ and $\xi \in W^{1,1}([0,T]) \cap C^1(0,T)$. Let $L^+, L^-$ be the maximal solutions to \eqref{eq:opt_SDE}, \eqref{eq:opt_SDE2}, let $\tau^{\pm} = \inf\{t>0, L^\pm(t)=0\}$ and $\tau=\tau^+ \wedge \tau^-$. Then $u \in C^{1,4}((0,\tau)\times \R)$ with $u(t,\cdot) \in \mathcal{U}$ for all $t\in [0,\tau)$.
\end{lemma} 

\begin{proof}
Without loss of generality, we assume that $u$ is smooth and obtain $L^\infty$ estimates from the PDE applied to the derivatives of $u$. This can be easily justified by considering solutions $u^\varepsilon$ to the equations with an additional viscosity $\varepsilon u_{xx}$ in the right-hand side, and noting that the bounds obtained from the arguments below are uniform in $\varepsilon$.  

Now we first note that the fact that $0 \leq u_x \leq 1$, $u_{xx} \geq 0$ is clear by \eqref{eq:ux}, \eqref{eq:uxx} and the maximum principle, and so is the fact that $u(t,\cdot)$, $u(t,1+\cdot)$ are even for all $t\geq 0$. In addition, we already know from Theorem \ref{thm:plp} that $u_{xx}(t,\cdot)$ is bounded for $t \in [0, \tau)$. We set $u_i := (\partial_x)^i u$ and observe that
\begin{equation} \label{eq:u3}
\left\{ \begin{array}{l} \partial_t u_3 =  \frac 32 u_3^2 u_1 + 3 u_3 u_2^2 + 2 u_4 u_2 u_1 + \frac{1}{4} u_5 u_1^2 - \dot{\xi}(t) \left(3 u_3 u_2  + u_1 u_4\right) \\
u_3(0,x) = (u^0)_3(x) , \;\;u_3(t,0)=0, \;\; u_3 \mbox{ bounded. } \end{array}
  \right. 
\end{equation}
One first checks that $\sup_{x\in\R} u_3(0,x) \leq 0$ implies $\sup_{x\in\R}u_3(t,x) \leq 0$, by a maximum principle argument. Since the only nonlinear term in the right hand side of \eqref{eq:u3} is $3 u_3^2 u_1 \geq 0$, the maximum principle implies that on $[0,  \tau)\times \R_+$,  
$$0 \geq u_3 \geq  - \|u_0\| \exp\left(6 \|u_2\|^2_\infty \tau + \|u_2  \|_\infty \int_0^\tau |\dot{\xi}(s)|ds \right).$$

Then one writes in a similar way the equation for $u_4$ (and then $u_5$, $u_6$), noting that this time they are linear with coefficients depending on $u_1, u_2, u_3$, (resp. $u_1$ to $u_4$, and $u_1$ to $u_5$) so that $u_4$, $u_5$ and $u_6$ also stay bounded for $t <\tau$.

Finally, from \eqref{eq:exopt}, \eqref{eq:ux}, \eqref{eq:uxx}, \eqref{eq:u3} one gets that boundedness of $u_1, \ldots, u_6$ implies continuity of $\partial_t u, \ldots, \partial_t u_4$, i.e. $u \in C^{1,4}([0,\tau)\times \R)$.
\end{proof}

%

\begin{lemma}\label{lem:stabU} Let $u_0 \in \mathcal{U}$, $\xi \in C([0,\infty])$ and $u$ be the solution to \eqref{eq:exopt}. Then, $u(t,\cdot) \in \mathcal{U}$ for all $t \geq 0$. 
\end{lemma}
\begin{proof} Let $u^{0,\varepsilon} \in \mathcal{U}$ be smooth approximations of $u^0$,  $\xi^\varepsilon$ be smooth approximations of $\xi$ and $u^\varepsilon$ be the unique smooth solution (cf.\ \cite{LSU67}) to 
\begin{equation}\label{eq:approx}
  \partial_t u^\varepsilon = \left( \varepsilon + \frac{1}{4}|u_x^\varepsilon|^2\right) u_{xx}^\varepsilon -  \frac{1}{2}|u_x^\varepsilon|^2 \dot{\xi}^\varepsilon(t), \;\;\; u(0,\cdot) = u^{0,\varepsilon}(\cdot).
\end{equation}
Since $u^\varepsilon$ is smooth, as in the proof of the previous lemma we may differentiate \eqref{eq:approx} and use the maximum principle to obtain that for each $\varepsilon>0$, $u^\varepsilon$ is $2$-periodic, symmetric in $x$ around $0$ and $1$, and $0 \leq u_x^\varepsilon \leq 1, u_{xxx}^\varepsilon \leq 0$ on $[0,+\infty) \times (0,1)$. Since $u^\varepsilon \to u$ uniformly and $\mathcal{U}$ is stable under uniform convergence, we can conclude.
\end{proof}

\begin{proposition} \label{prop:opt1}
Assume that $u^0_{xx} (0)< \infty$, then $u_{xx}(t,0) = \frac{1}{L^+(t)}$ for $t \leq \tau^+ := \inf\left\{s > 0, L^+(s) = 0\right\}$.
\end{proposition}
\begin{proof}
In the case of  $\xi$ $\in$ $C^1$ and $u\in C^{1,4}$ with $u(t,\cdot) \in \mathcal{U}$ for all $t\ge0$, the result follows from differentiating \eqref{eq:exopt} twice
\begin{align} 
	\partial_t u_x &=  \frac{1}{4} u_{xxx} u_x^2 + \frac{1}{2} u_{xx}^2 u_x - \dot{\xi}(t) u_{xx} u_x,\label{eq:ux}\\
    \partial_t u_{xx} &= \frac{1}{4} u_{xxxx} u_x^2 +  \frac{3}{2} u_{xxx} u_{xx} u_x +  \frac{1}{2} u_{xx}^3 - \dot{\xi}(t) \left(u_{xx}^2 + u_{xxx} u_x\right),\label{eq:uxx}
\end{align}
and noting that $u_x(t,0)=u_{xxx}(t,0)=0$ for all $t\ge 0$.


Let $\xi^\eta \in W^{1,1}([0,T]) \cap C^1(0,T)$ with $\xi^\eta \uparrow \xi$, $\xi^\eta(0)=\xi(0)$.
Further let $u^{0,\eta} \in C^6_b \wedge \mathcal{U}$ with $u^{0,\eta} \to u^0$ uniformly, $u^{0,\eta}(0)=u^0(0)$, $u^{0,\eta} \leq u^0$ and such that $u_{xx}^{0,\eta}(0) \uparrow u^0_{xx}(0)$. Also assume that $u^{0,\eta}_{xx}(1)$ is chosen small enough that if $L^{+,\eta}$, $L^{-,\eta}$ are the solution to \eqref{eq:reflected} driven by $\xi^\eta$ and starting from $\frac{1}{u_{xx}^{0,\eta}(0)}$, $-\frac{1}{u_{xx}^{1,\eta}(0)}$, the hitting times of $0$ satisfy $\tau^{-,\eta} >  \tau^{+,\eta})$.  Let $u^\eta$ be the solution to \eqref{eq:exopt} driven by $\xi^\eta$ and starting from $u^{0,\eta}$. By Lemma \ref{lem:2}, for $t \in [0,T]$, 
 $$u^\eta_{xx}(t,0) = \frac{1}{L^{+,\eta}(t)}.$$ 
 We note that $L^{+,\eta}(t) \uparrow_{\eta\to 0} L^+(t)$ uniformly in $[0,\tau^+]$ and, by Lemma \ref{lem:stabU}, $u_{xx}^\eta(t,0) =  \sup_{\delta \in (0,1)} \frac{u^\eta(t,\delta)  - u^\eta(t,0)}{\delta^2}$. Finally, from \eqref{eq:FpX} it follows that $u^\eta \uparrow u$ with $u^\eta(t,0) = u(t,0)(=u^0(0))$, and we get
$$u_{xx}(t,0) =  \sup_{\delta \in(0,1)} \sup_{\eta >0} \frac{u^\eta(t,\delta) - u^\eta(t,0)}{\delta^2}=\sup_{\eta >0} u_{xx}^\eta(t,0) = \frac{1}{L^{+}(t)}.$$
\end{proof}

\begin{lemma} \label{lem:Compux}
  Let $\xi\in C([0,T])$, $u_{0}\in(BUC\cap W^{1,1})([0,2])$ periodic and $u$ be the corresponding viscosity solution to \eqref{eq:exopt}. Then $v=\partial_{x}u$ is the pathwise entropy solution\footnote{For a theory of pathwise entropy solutions to \eqref{eq:sscl} we refer to \cite{GS16}.} to 
  \begin{equation}\begin{aligned}\label{eq:sscl}
  dv +\frac{1}{2}\partial_{x}v^{2}\circ d\xi(t)& =\frac{1}{12}\partial_{xx}v^{[3]}dt\\
  v(0) & =\partial_{x}u_{0}.
  \end{aligned}\end{equation}
  Let $u_{0}^{1},u_{0}^{2}\in(BUC\cap W^{1,1})([0,2]) \cap \mathcal{U}$ and $u^{1},u^{2}$ be the corresponding viscosity solutions to \eqref{eq:exopt} such that $\partial_{x}u_{0}^{1}\ge\partial_{x}u_{0}^{2}$ a.e. on $(0,1)$. Then for all $t \geq 0$,
    \[
    \partial_{x}u^{1}(t,\cdot)\ge\partial_{x}u^{2}(t,\cdot)\quad\text{a.e. on } (0,1).
    \]
\end{lemma}
\begin{proof}
  We consider $u_{0}^{n}$ smooth, periodic such that $u_{0}^{n}\to u_{0}$ uniformly and in $W^{1,1}([0,2])$. Further let $\xi^{n}$ smooth with $\xi^{n}\to\xi$ uniformly. For $\ve>0$ let $u^{\ve,n}$ be the unique classical solution to 
  \begin{equation}\begin{aligned}\label{eq:u_approx}
  du^{\ve,n} & =\left(\ve u_{xx}^{\ve,n}+\frac{1}{4}|u_{x}^{\ve,n}|^{2}u_{xx}^{\ve,n}\right)dt-\frac{1}{2}(u_{x}^{\ve,n})^{2}\dot{\xi}^{n}(t)\\
  u^{\ve,n}(0) & =u_{0}^{n}.
  \end{aligned}\end{equation}
  Then $v^{\ve,n}:=\partial_{x}u^{\ve,n}$ is the unique solution to 
  \begin{equation}\begin{aligned}\label{eq:v_approx}
  dv^{\ve,n} & =\left(\ve v_{xx}^{\ve,n}+\frac{1}{12}\partial_{x}(v^{\ve,n}){}^{3}\right)dt-\frac{1}{2}\partial_{x}(v^{\ve,n})^{2}\dot{\xi}^{n}(t)\\
  v^{\ve,n}(0) & =\partial_{x}u_{0}^{n}.
  \end{aligned}\end{equation}
  By stability of viscosity solutions we have $u^{\ve,n}\to u^{n}$ uniformly and $v^{\ve,n}\to v^{n}$ in $C([0;T];L^1)$ by \cite{P02}, where $u^{n}$ is the viscosity solution to \eqref{eq:u_approx} and $v^{n}$ is the kinetic solution to \eqref{eq:v_approx} with $\ve=0$ respectively. By Theorem \ref{thm:app_wp} we have $u^{n}\to u$ uniformly and by \cite[Theorem 2.3, Proposition 2.5]{GS16} we have $v^{n}\to v$ in $C([0,T];L^{1})$, where $u$ is the viscosity solution to \eqref{eq:exopt} and $v$ is the kinetic solution to \eqref{eq:sscl}.
  
  Let now $u_{0}^{1},u_{0}^{2}\in(BUC\cap W^{1,1})([0,2])\cap \mathcal{U}$ with $\partial_{x}u_{0}^{1}\ge\partial_{x}u_{0}^{2}$ a.e. on $(0,1)$. As above, consider the respective approximations $u^{1,\ve,n}$, $u^{2,\ve,n}$, with $u^{1,n}_0, u^{2,n}_0$ smooth elements of $\mathcal{U}$ with $\partial_x u^{1,n}_0 \ge u^{2,n}_0$ in $[0,1]$. Then, as in Lemma \ref{lem:stabU}, $u^{1,\ve,n}(t,\cdot), u^{2,\ve,n}(t,\cdot)$ $\in$ $\mathcal{U}$ for all $t\geq0$. Note that for $u$ $\in C^1 \cap \mathcal{U}$, $\partial_x u(0)=\partial_x u(1)=0$. Hence, $\partial_{x}u^{1,\ve,n}(t,\cdot)\ge\partial_{x}u^{2,\ve,n}(t,\cdot)$ on $[0,1]$ by the comparison principle for \eqref{eq:v_approx} with Dirichlet boundary conditions on $(0,1)$. Taking limits implies the claim.  
\end{proof}

\begin{proposition} \label{prop:opt2}
The map $t\mapsto u_{xx}(t,0) \in (0,\infty]$ is continuous.
\end{proposition}
\begin{proof}
First note that $t\mapsto u_{xx}(t,0)$ is lower semicontinuous as supremum of continuous functions by \eqref{eq:uxx0}, and taking also into account Proposition \ref{prop:opt1}, we only need to prove that
\begin{equation} \label{eq:cont}
t_n \nearrow t ,\;\;\; u_{xx}(t_n,0) \to +\infty \;\;\Rightarrow \;\; u_{xx}(t,0)=+\infty.
\end{equation}

We fix $M>0$ and let $u^n$ be solutions to \eqref{eq:exopt} but starting from data $u^{t_n,n}$ at time $t_n$, where $u^{t_n,n} \in \mathcal{U}$ is such that $u^{t_n,n}_{xx}(0) = M$ and $u^{t_n,n}_x \leq u_x(t_n,\cdot)$ (this is possible at least for $n$ large enough). By Proposition \ref{prop:opt1}, $u^n_{xx}(s,0) = \frac{1}{L^{+,n}(s)}$ for $s \in [t_n, \tau^{+,n})$, where $$dL^{+,n}(s) = -\frac{1}{2 L^{+,n}(s)}ds + d\xi(s),\;\;\;\; L^{+,n}(t_n) = M^{-1}$$
and $\tau^{+,n} = \inf\left\{s > t_n, L^{+,n}(s) = 0\right\}$. By Lemma \ref{lem:2} one has $\tau^{+,n} > t$ for $n$ large enough, and, clearly, $\lim_{n\to \infty} L^{+,n}(t) = M^{-1}$. Since $u_{xx}(t,0) \geq u^n_{xx}(t,0)$ by Lemma \ref{lem:Compux}, it follows that $u_{xx}(t,0) \geq M$. Since $M$ was arbitrary, this proves \eqref{eq:cont}.
\end{proof}

\appendix

\section{Stochastic viscosity solutions}\label{sec:visc_soln}

%

In this section we briefly recall the definition and main properties of stochastic viscosity solutions to fully nonlinear SPDE of the type
\begin{equation}\begin{aligned}\label{eq:app_SPDE}
d u +  \frac{1}{2} |D u|^2 \circ d{\xi}(t) &= F(t,x,u,Du,D^2u)dt  \quad\text{in }\R^N\times(0,T] \\
u(0,\cdot) &= u_0\quad\text{on }\R^N\times \{0\},
\end{aligned}\end{equation}
where $u_0 \in BUC(\R^N)$, $F\in C([0,T]\times\R^N\times\R\times\R^N\times S^N)$ and $\xi$ is a continuous path.

We recall from \cite[Theorem 1.2, Theorem 1.3]{FGLS16}
\begin{theorem}\label{thm:app_wp}
  Let $u_0,v_0 \in BUC(\R^N)$, $T>0$, $\xi,\z\in C^1_0([0,T];\R)$ and assume that Assumption \ref{asn:F} holds. If $u\in BUSC([0,T]\times\R^N)$, $v\in BLSC([0,T]\times\R^N)$ are viscosity sub- and super-solutions to \eqref{eq:app_SPDE} driven by $\xi,\z$ respectively, then,
  \begin{align} \label{eq:Fst}
    \sup_{[0,T]\times \R^N} (u-v) \le \sup_{\R^N}(u_0-v_0)_+ + \Phi(\|\xi-\z\|_{C([0,T])}) ,
  \end{align}
  where $\Phi$ depends only on $T$, the sup-norms and moduli of continuity of $u_0, v_0$ and the quantities appearing in Assumption \ref{asn:F} \revP{(2)-(3)-(5)}, is non-decreasing and such that $\Phi(0^+)=0$.
  In particular, the solution operator 
    $$S: BUC(\R^N) \times C^1_0([0,T];\R^N) \to BUC([0,T]\times\R^N)$$ 
  admits a unique continuous extension to 
  $$S: BUC(\R^N) \times C^0_0([0,T];\R^N) \to BUC([0,T]\times\R^N).$$
  We then call $u=S^\xi(u_0)$ the unique viscosity solution to \eqref{eq:app_SPDE}. One then has  
  \begin{equation}\label{eq:app_ctn_visc_soln}
    \|S^{\xi}(u_0)-S^{\z}(v_0)\|_{C([0,T]\times\R^N)} \le \|u_0-v_0\|_{C(\R^N)}+\Phi\left(\|\xi -\z\|_{C([0,T])}\right).
  \end{equation}
  In the case where $F=F(p,X)$ only depends on its last two arguments, the estimate simplifies to 
   \begin{align} \label{eq:FpX}
    \sup_{[0,T]\times \R^N} (u-v) \le \sup_{x,y \in \R^N} \left( u_0(x)-v_0(y) - \frac{|x-y|^2}{\sup_{s \in [0,T]} (\xi(s)-\zeta(s))} \right)
  \end{align}
  (with convention $0/0=0$, $1/0=+\infty$).
\end{theorem}

%

\bibliographystyle{plain}
\def\cprime{$'$}

\end{document}